\documentclass[11pt,a4paper]{article}
\usepackage{amsmath, amsfonts, amssymb,amsthm}
\usepackage{a4wide}
\usepackage{parskip}
\usepackage{enumitem}
\usepackage{xcolor}
\usepackage{pstricks}
\usepackage{hyperref}
\usepackage{latexsym}
\usepackage{cite}

\usepackage[left=1in, right=1in,top=1in,bottom=1in]{geometry}
\allowdisplaybreaks

\newcommand{\be} {\begin{equation}}
	\newcommand{\ee} {\end{equation}}
\newcommand{\bea} {\begin{eqnarray}}
	\newcommand{\eea} {\end{eqnarray}}
\newcommand{\Bea} {\begin{eqnarray*}}
	\newcommand{\Eea} {\end{eqnarray*}}

\numberwithin{equation}{section}

\newtheorem{definition}{Definition}[section]
\newtheorem{theorem}{Theorem}[section]
\newtheorem{rem}{Remark}[section]
\newtheorem{lemma}{Lemma}[section]
\newtheorem{prop}{Proposition}[section]

\numberwithin{equation}{section}

\begin{document}
\setlength{\abovedisplayskip}{3pt}
	\setlength{\belowdisplayskip}{3pt}
	\date{}
	\title{Normalized Solutions of the $L^2$-Supercritical NLS Equation on Noncompact Metric Graphs with a Short-Range Potential}

	\author{ {\bf Archana Prajapati$\,^{1,}$\footnote{e-mail: {\tt matharchana3@gmail.com}}}, {\bf Divya Goel$\,^{1,}$\footnote{e-mail: {\tt divya.mat@iitbhu.ac.in}}} \\ $^1\,$Department of Mathematical Sciences, Indian Institute of Technology (BHU),\\ Varanasi 221005, India.}
	
	\maketitle

		\begin{abstract}		
		We are concerned with the existence of normalized solutions to the $L^2$-supercritical nonlinear Schr\"odinger equation on a noncompact metric graph $G$,
		\[	\begin{cases}
			-u''+W(x)u+\lambda u=\chi(x)|u|^{p-2}u, & \text{on every edge } e \text{ of } G,\\[2mm]
			\displaystyle\sum_{e\succ v}u'_e(v)=0, & \text{at every vertex } v\in V,
		\end{cases}
		\]
		under the mass constraint $\int_G |u|^2\,dx=\mu>0$, where $\lambda$ arises as a Lagrange multiplier. Here $p>6$, $\chi$ is the characteristic function of the compact core $\mathcal K$, so that the nonlinearity is localized, and the potential $W$ is bounded, nonnegative and vanishing at infinity along every unbounded edge, with $\int_{1}^{\infty}xW(x)\, dx<\infty$ on at least one of them. For every $\mu>0$ we obtain a positive solution with $\lambda>0$, arising as a constrained critical point at a strictly positive energy level. Our approach combines a uniform mountain-pass geometry for a family of approximating functionals, the monotonicity trick with Morse index type information, and a blow-up analysis ruling out the divergence of the Lagrange multipliers. A key point is the strict positivity of the multiplier, for which we exhibit a potential of inverse-square decay admitting a zero-energy $L^2$ solution on a half-line. To the best of our knowledge, this is the first existence result for normalized solutions of the $L^2$-supercritical NLS equation on a noncompact metric graph in the presence of an external potential.
		\\ [3mm]

		
		\medskip
		
	\end{abstract}
	
	\section{Introduction}
	In this paper, we study the existence of bound states of prescribed mass for the $L^2$-supercritical NLS equation with a potential vanishing at infinity and localized nonlinearity on a noncompact metric graph $G$,
	\begin{equation}\label{b1:1}
	-u'' + W(x)u + \lambda u = \chi(x)|u|^{p-2}u, 	\qquad \text{on every edge } e \text{ of } G,
	\end{equation}
	coupled with continuity at the vertices and the Kirchhoff conditions 	\begin{equation}\label{b1:2}
	\sum_{e\succ v} u'_e(v)=0, \qquad \text{at every vertex } v\in V,
	\end{equation}
	the derivatives being taken outward from $v$. Here $p>6$, $G$ is a connected noncompact metric graph, $W: G\to\mathbb{R}$ is a given potential, $\lambda\in\mathbb{R}$ is not prescribed a priori but arises as a Lagrange multiplier, and $\chi$ is the characteristic function of the compact core $\mathcal K$ of $G$, that is, of the metric subgraph formed by all bounded edges \cite{adami2016threshold}. A solution of \eqref{b1:1}--\eqref{b1:2} subject to \begin{equation}\label{b1:3}
	\int_G |u|^2\,dx = \mu > 0
	\end{equation}
	is called a \emph{normalized solution}; such solutions are the critical points of the energy functional 	\begin{equation}\label{b1:4}
	E_W(u,G) = \frac12\int_G |u'|^2\,dx + \frac12\int_G W(x)|u|^2\,dx - \frac1p\int_{\mathcal K} |u|^p\,dx
	\end{equation}
	constrained to $H^1_\mu(G)=\{u\in H^1(G): \|u\|_{L^2(G)}^2=\mu\}$. Stationary solutions of \eqref{b1:1} correspond to standing waves $\psi(t,x)=e^{i\lambda t}u(x)$ of the time-dependent equation
	\begin{align*} 	i\,\partial_t\psi = -\partial_{xx}\psi + W(x)\psi - {\chi(x)}|\psi|^{p-2}\psi \qquad \text{on } \mathbb{R}\times G, \end{align*}
	and the constraint \eqref{b1:3} expresses the conservation of mass, or charge, of the system; in quantum mechanics the normalization $\|u\|_{L^2(G)}^2=1$ says that the total probability equals one. The variational study of \eqref{b1:1} rests on the Gagliardo--Nirenberg inequality on graphs \cite{adami2016threshold, tentarelli2016nls}, which singles out $p=6$ as the $L^2$-critical exponent. The range $p\in(2,6)$ is the $L^2$-subcritical regime, $p=6$ the critical one, and $p\in(6,\infty)$ the $L^2$-supercritical regime considered here.

	Nonlinear Schr\"odinger (NLS) equations on metric graphs describe wave propagation along network-shaped structures, and they occur in optical waveguides, photonic crystals and nanostructures. For the physical background, we refer to the surveys \cite{noja2014nonlinear, berkolaiko2013introduction}, and for the connection with the theory of quantum graphs to \cite{kairzhan2022standing, berkolaiko2013introduction}. When the nonlinear response of the medium is confined to a bounded region of the network, as happens in optical fibers with a localized nonlinear response or in Bose--Einstein condensates trapped in nonuniform media \cite{noja2014nonlinear, gnutzmann2011stationary}, the nonlinearity acts only on a compact part of the graph.

	\medskip
	\noindent\textbf{Known results.}
	The variational theory of NLS on metric graphs was started in the $L^2$-subcritical range $2<p<6$, where the energy is bounded below on the mass constraint and minimization applies. In a series of works, Adami, Serra and Tilli \cite{adami2015nls, adami2016threshold, adami2017negative} proved that the existence of ground states depends on both the topology and the metric of $G$, a phenomenon with no Euclidean counterpart. In \cite{adami2015nls}, the authors established a purely topological obstruction: if $G$ satisfies the assumption $(H)$, then the infimum of the energy equals the corresponding infimum on the real line and is never attained, so no ground state exists for any mass. Graphs violating $(H)$, such as the tadpole or graphs with a terminal edge, can admit ground states, and threshold phenomena in the mass appear \cite{adami2016threshold}. In \cite{adami2017negative}, the authors settled the mass-critical case $p=6$, where the governing parameter is the critical mass associated with the best constant in the $L^2$-critical Gagliardo--Nirenberg inequality on $G$, and ground states, when they exist, have negative energy. Multiplicity of bound states in the subcritical regime was obtained in \cite{adami2019multiple}. The compact case, where the embedding $H^1(G)\hookrightarrow L^p(G)$ is compact and the analysis simplifies, was treated in \cite{cacciapuoti2018variational, dovetta2018existence}; combined nonlinearities and local minimizers were studied in \cite{pierotti2022ground, pierotti2021local}, and periodic graphs in \cite{pankov2018nonlinear}.

	A second line of work, closer to the present setting, concerns nonlinearities acting only on the compact core $\mathcal K$, as in \eqref{b1:1}. Along the half-lines the equation is then linear, so these models behave differently from the extended case. In the subcritical regime, existence and nonexistence of ground states were established by Tentarelli \cite{tentarelli2016nls} and by Serra and Tentarelli \cite{serra2016bound}, the thresholds being expressed through the measure $|\mathcal K|$ rather than the topology alone; in \cite{serra2016bound}, the authors obtained multiple bound states at negative energy levels by genus theory. The critical case $p=6$ with localized nonlinearity was studied by Dovetta and Tentarelli \cite{DovettaTentarelli2019, dovetta2020ground}, first on the tadpole graph and then in general, where existence is governed by a reduced critical mass depending on both metric and topological features of $G$. Peaked and low-action solutions on graphs with terminal edges \cite{dovetta2020peaked}, edge-localized states in the large-mass limit \cite{berkolaiko2021edge}, and peaked bound states \cite{chen2024existence} have also been constructed. For the $\mathcal T$-graph, a classification of the positive solutions of $-u''+\lambda u=|u|^{p-2}u$, together with uniqueness of action ground states, was obtained in \cite{agostinho2024classification}.

	The supercritical regime $p>6$ is different in nature. Here $E_W(\cdot,G)$ is unbounded below on $H^1_\mu(G)$, minimization is unavailable, and one looks for critical points of saddle type. In \cite{jeanjean1997existence}, Jeanjean solved the corresponding problem in $\mathbb{R}^N$ using the dilations $u\mapsto s^{N/2}u(s\cdot)$, which preserve the mass and produce a bounded Palais--Smale sequence through the Pohozaev identity; the method was later carried over to systems, combined nonlinearities and the Sobolev critical case \cite{bartsch2017natural, soave2020normalized, soave2020normalizedsobolev}, and to bounded domains \cite{noris2014existence, pierotti2017normalized}. On a metric graph these dilations are unavailable, since rescaling changes the lengths of the bounded edges; the graph is not scale invariant, there is no Pohozaev identity, and the arguments resting on it do not carry over. In \cite{chang2023normalized}, Chang, Jeanjean and Soave overcame this for compact graphs, combining a parametrized mountain-pass scheme with the monotonicity trick and second-order Morse-type information on the Palais--Smale sequences supplied by the abstract theorem of \cite{borthwick2024bounded}, the boundedness of the multipliers being recovered by a blow-up analysis in the spirit of \cite{esposito2011pointwise}. In \cite{borthwick2023normalized}, the authors carried this to noncompact graphs with the nonlinearity localized on the compact core, obtaining for every $\mu>0$ a positive normalized solution with $\lambda>0$ at a positive energy level; there the noncompactness is compensated by the vanishing of the nonlinearity along the half-lines, where the equation reduces to $-u''+\lambda u=0$ and solutions decay exponentially.
    
	Recently, in \cite{carrillo2026blow}, Carrillo, De Coster, Galant, Jeanjean and Troestler developed a general blow-up analysis for NLS equations on metric graphs, valid for possibly sign-changing solutions and for sequences of potentials bounded in $L^\infty$.  For sequences of solutions with $\lambda_n\to+\infty$ and uniformly bounded Morse index, their Theorem 1.3 gives a global exponential decay estimate of the type of our Proposition \ref{bp:4}; when in addition the mass is fixed and $p\neq6$, their Corollary 1.4 yields the boundedness of $\{\lambda_n\}$, which is the conclusion of Step 1 in the proof of Theorem \ref{bt:2}. Their Proposition 4.8 gives, for finite Morse index solutions on a graph with a half-line, a lower bound on $\lambda$ in terms of the behaviour of $W$ at infinity, of which $\lambda\ge0$ under $(W3)$ is a particular case; our Lemma \ref{bl:3} serves the same purpose for the approximate Morse index of Palais--Smale sequences. They also study the $L^2$ norms of solutions as $\lambda\to0^+$, in settings where the potential vanishes on the half-lines and the solutions there are explicit.

	Much less is known when an external potential is present. In the Euclidean setting the theory is well developed. Ikoma and Miyamoto \cite{ikoma2020stable} treated the mass-subcritical case by minimization. Bartsch, Molle, Rizzi, and Verzini \cite{bartsch2021normalized} obtained normalized solutions in the mass-supercritical, Sobolev-subcritical range for potentials $V\ge0$ vanishing at infinity, by a new min-max argument replacing the scaling geometry that the potential destroys. Molle, Riey, and Verzini \cite{molle2022normalized} studied negative potentials, obtaining two solutions for small mass together with a nonexistence result. On metric graphs, potentials have been considered mostly outside the supercritical regime. In \cite{cacciapuoti2015ground}, the authors studied ground states and their orbital stability on starlike graphs with potentials. The existence of ground states for attractive potentials on noncompact graphs was established in \cite{AdamiGalloSpitzkopf2025}, and that of multiple positive bound states in \cite{liu2026multiple}, again in the subcritical regime; magnetic potentials were considered in \cite{cangiotti2026nonlinear}. For the fixed-frequency problem, Akduman and Pankov \cite{akduman2019nonlinear} proved the existence of nontrivial solutions on infinite graphs for potentials growing at infinity, in which case the spectrum of $-\tfrac{d^2}{dx^2}+W$ is discrete and compactness is restored; see also \cite{pankov2018nonlinear} for periodic potentials.

	On metric graphs, the available results thus concern either the subcritical regime or the fixed-frequency problem. To the best of our knowledge, no existence result is available for normalized solutions of the $L^2$-supercritical NLS equation on a noncompact metric graph in the presence of an external potential. The present article addresses this question.

	\medskip

	\noindent\textbf{Assumptions and main result.}
	Throughout we assume that $W:G\to\mathbb{R}$ satisfies
	\begin{itemize}
		\item[$(W1)$] $W\in L^\infty(G)$;
		\item[$(W2)$] $W(x)\ge 0$ for a.e.\ $x\in G$;
		\item[$(W3)$] $\lim_{x\to\infty} W_e(x)=0$ for every unbounded edge $e$ of $G$;
		\item[$(W4)$] $\displaystyle\int_1^{+\infty}x{W_\ell}(x)\,dx<+\infty$ {for} at least one unbounded edge $\ell\simeq[0,+\infty)$ of $G$.
	\end{itemize}
	Assumption $(W4)$ is a short-range condition of the kind used in one-dimensional scattering theory. It is required on a single unbounded edge, and by $(W1)$ it restricts only the behaviour of $W$ at infinity; it holds for every compactly supported potential, for every exponentially decaying one, and whenever $W(x)=O(x^{-2-\delta})$ for some $\delta>0$.
	\begin{theorem}\label{bt:2}
	Let $G$ be a noncompact metric graph with finitely many edges and nonempty compact core $\mathcal K$, let $p>6$, and let $W$ satisfy $(W1)$--$(W4)$. Then, for every $\mu>0$, problem \eqref{b1:1}--\eqref{b1:3} admits a positive solution $u$, associated with a Lagrange multiplier $\lambda>0$, which is a constrained critical point of $E_W(\cdot,G)$ on $H^1_\mu(G)$ at a strictly positive energy level.
	\end{theorem}

	\medskip
	\noindent\textbf{Strategy and main difficulties.}
	We introduce the parametrized family 	\begin{align*} 	E_{W,\tau}(u,G) = \frac12\int_G |u'|^2\,dx + \frac12\int_G W(x)|u|^2\,dx - \frac{\tau}{p}\int_{\mathcal K} |u|^p\,dx , 	\qquad \tau\in\Bigl[\tfrac12,1\Bigr], 	\end{align*} 	show that it has a mountain-pass geometry on $H^1_\mu(G)$ uniformly in $\tau$, apply the monotonicity trick of \cite{borthwick2024bounded} to produce, for almost every $\tau$, a mountain-pass solution $u_\tau$ with $m(u_\tau)\le2$, and then let $\tau\to1^-$. The potential enters at every step; at three points the autonomous arguments of \cite{borthwick2023normalized} provide no guidance.

	The first is the mountain-pass geometry (Lemma~\ref{bl:2}). By $(W2)$ the term $\tfrac12\int_G W|u|^2$ leaves the coercivity estimate intact, but the low endpoint $w_1$ must be built as a bump spreading along a half-line, its energy being pushed below the barrier level only because $(W3)$ forces $\int_G W|w_1|^2\to0$; the same mechanism reappears in Lemma~\ref{bl:3}.

	The second is the strict positivity of $\lambda_\tau$ (Proposition~\ref{bp:1}). When $W\equiv0$ and $\lambda_\tau=0$, the equation on a half-line reduces to $-u_\tau''=0$, and $u_\tau\in L^2$ forces $u_\tau\equiv0$; here $u_\tau$ would solve $-u_\tau''+Wu_\tau=0$, and the variational scheme yields only $\lambda_\tau\ge0$. The sign condition $(W2)$ does not improve this: it gives $\inf\sigma\bigl(-\tfrac{d^2}{dx^2}+W\bigr)\ge0$ but no spectral gap, since under $(W3)$ the essential spectrum is $[0,+\infty)$. What must be ruled out is a zero-energy $L^2$ solution on a half-line, and such solutions exist under $(W1)$--$(W3)$: for $W(x)=2(1+x)^{-2}$, the function $(1+x)^{-1}$ is one. Assumption $(W4)$ excludes them (Lemma~\ref{bl:7}), whence $\lambda_\tau>0$; the same argument gives $\hat\lambda>0$ in Step 4 of the proof of Theorem~\ref{bt:2}.

	The third is the limit $\tau\to1^-$, where boundedness of $\{u_n\}$ reduces to that of $\{\lambda_n\}$. If $\lambda_n\to+\infty$, a blow-up analysis in the spirit of \cite{chang2023normalized, esposito2011pointwise} (Lemmas~\ref{bl:4}--\ref{bl:6}, Propositions~\ref{bp:2}--\ref{bp:4}) shows that at most two concentration points occur and that $u_n$ decays exponentially away from them; since $p>6$, the total mass then vanishes, contradicting $\|u_n\|_{L^2(G)}^2=\mu>0$.

    In the autonomous case $\lambda_\tau>0$ is immediate; here it fails in general, a zero-energy $L^2$ mode of $-\tfrac{d^2}{dx^2}+W$ on a half-line being an obstruction that no variational argument removes. This obstruction is invisible when $W$ vanishes on the half-lines, where the solutions are the explicit exponentials $ce^{-\sqrt\lambda\,x}$. Assumption $(W4)$ and Lemma \ref{bl:7} settle the general case, and we pin down where each of $(W2)$, $(W3)$ and $(W4)$ is used.
	\medskip
	\noindent\textbf{Organization of the paper.}
	Section~\ref{bs:1} collects the functional framework on metric graphs. In Section~\ref{bs:2} we establish the existence of solutions to the approximating problems. Section~\ref{bs:3} is devoted to the blow-up analysis and to the proof of Theorem~\ref{bt:2}.
	
	\section{Preliminaries}\label{bs:1}
	Throughout this work, $G = (E, V)$ denotes a connected, noncompact metric graph, with $E$ its edge set and $V$ its vertex set. We assume $G$ has finitely many edges and a nonempty compact core $\mathcal{K}$. Denote by $\{e_1,\ldots,e_{m_1}\}$ the bounded edges, by $\{\ell_1,\ldots,\ell_{m_2}\}$ the half-lines, and by $\{v_1,\ldots,v_{m_3}\}$ the vertices of $G$. Each bounded edge $e$ is identified with $[0,\ell_e]$, $\ell_e>0$ being its length, while each half-line is identified with $[0,+\infty)$. The compact core $\mathcal{K}$ is the union of all bounded edges.  We view a function $u$ on $G$ as a collection $\{u_e\}$, each $u_e$ defined on $(0,\ell_e)$ for $e\in E$, rather than prescribed only at the vertices as in discrete models. Fix $1 \le p \le \infty$. We say $u\in L^p(G)$ if $u_e \in L^p(0,\ell_e)$ for every $e \in E$ and 
    \begin{align*} \|u\|_{L^p(G)}^p = \sum_{e \in E} \|u_e\|_{L^p(e)}^p < \infty, \end{align*} which yields the direct sum decomposition $L^p(G) = \bigoplus_{e \in E} L^p(e).$  The Sobolev space $H^1(G)$ is the set of continuous functions $u$ on $G$ with $u_e \in H^{1}(I_e)$ for each $e \in E$ and \begin{align*} \|u\|_{H^{1}(G)}^2 = \sum_{e \in E}\left( \|u_e\|_{L^2(e)}^2 +\|u_e'\|_{L^2(e)}^2\right) < \infty. \end{align*} Viewing $H^1(G)$ and $L^2(G)$ as infinite dimensional spaces, we suppose the continuous embedding $H^1(G)\hookrightarrow L^2(G)$ has norm at most $1$ and identify $H^1(G)$ with its image in $L^2(G)$. For each $\mu>0$, put \begin{align*} S_{\mu} =\bigl\{u\in H^1(G):\, \|u\|_{{L}^{2}(G)}^2=\mu\bigr\}. \end{align*} Then $S_{\mu}$ is a smooth submanifold of $H^1(G)$ of codimension $1$, whose tangent space at $u \in S_{\mu}$ is the closed codimension $1$ subspace 
    \begin{align*} T_uS_{\mu} = \left\{w \in H^1(G) : \int_G uw \, dx =0 \right\}. \end{align*}  Hereafter, $\|\cdot\|_*$ and $\|\cdot\|_{**}$ denote the operator norms of $\mathcal{L}(H^1(G), \mathbb{R})$ and of $\mathcal{L}(H^1(G), \mathcal{L}(H^1(G), \mathbb{R}))$, respectively.
\begin{definition}\label{bd:1}
	Let $f : H^1(G) \to \mathbb{R}$ be a $C^2$-functional and $\alpha \in (0,1]$. We say $f'$ and $f''$ are $\alpha$-H\"older continuous on bounded sets if, for every $R>0$, there is $M = M(R)>0$ such that, for all $u_1, u_2 \in B(0,R)$,
	\begin{align*}\|f'(u_1)-f'(u_2)\|_{*} \le M\|u_1-u_2\|^{\alpha}_{H^1(G)}\end{align*}
	and
	\begin{align*}\|f''(u_1)-f''(u_2)\|_{**} \le M\|u_1-u_2\|^{\alpha}_{H^1(G)}.\end{align*}
\end{definition}

\begin{definition}\label{bd:2}
	Let $f$ be a $C^2$-functional on $H^1(G)$. For each $u \in H^1(G)$, define the continuous bilinear map 	
    \begin{align*} 
    D^2{f(u)} = f''(u) - \frac{f'(u) \cdot u}{\|u\|^2_{L^2(G)}}(\cdot, \cdot)_{L^2(G)}, 	
    \end{align*} 	
    where $(\cdot,\cdot)_{L^2(G)}$ is the inner product of $L^2(G)$.
\end{definition}

\begin{definition}\label{bd:3}
	Let $f$ be a $C^2$-functional on $H^1(G)$. For $u\in S_{\mu}$ and $\theta>0$, the $\theta$-approximate Morse index of $u$ is
	\begin{align*}
	\tilde{m}_{\theta}(u) = \sup \left\{ \dim Y :
	Y \subset T_{u}S_{\mu} \text{ is a subspace and } D^{2}f(u)[\varphi,\varphi] < -\theta \|\varphi\|^{2}_{H^1(G)}, \ \forall\,\varphi\in Y\setminus\{0\}
	\right\}.
	\end{align*}
	If $u$ is a critical point of $f|_{S_{\mu}}$ and $\theta=0$, then $\tilde{m}_{0}(u)$ is the Morse index of $u$ on $S_{\mu}$.
\end{definition}

We recall the following abstract result from \cite{borthwick2024bounded}.
\begin{theorem}\label{bt:1}
	Let $I\subset(0,+\infty)$ be an interval and $\{F_{\tau}\}_{\tau\in I}$ a family of $C^{2}$-functionals on $H^1(G)$ of the form
	\begin{align*}F_{\tau}(u)=A(u)-\tau B(u),\qquad \tau\in I,\end{align*}
	with $B(u)\ge 0$ for every $u\in H^1(G)$ and
	\begin{equation}\label{b:1}
		\text{either }A(u)\to +\infty \text{ or } B(u)\to +\infty
		\quad\text{as}\quad \|u\|_{H^1(G)}\to +\infty.
	\end{equation}
	Suppose $F_{\tau}'$ and $F_{\tau}''$ are $\alpha$-H\"older continuous on bounded subsets of $H^1(G)$ for some $\alpha\in(0,1]$. Let $w_1,w_2\in S_{\mu}$ be
	independent of $\tau$, and set
	\begin{align*}\mathbb{P}= \{ P\in C([0,1],S_{\mu}) :\, P(0)=w_1,\ P(1)=w_2 \}.\end{align*}
	Assume
	\begin{align*}	c_{\tau} = \inf_{P\in\mathbb{P}} \max_{t\in[0,1]} F_{\tau}(P(t)) > \max\{F_{\tau}(w_1),F_{\tau}(w_2)\}, \qquad \tau\in I. 
    \end{align*}
	Then, for almost every $\tau\in I$, there exist $\{u_n\}\subset S_{\mu}$ and $\theta_n\to 0^{+}$ such that, as $n\to\infty$,
	\begin{enumerate}
		\item[(i)] $F_{\tau}(u_n)\to c_{\tau}$;
		\item[(ii)] $\|F_{\tau}'|_{S_{\mu}}(u_n)\|_{*}\to 0$;
		\item[(iii)] $\{u_n\}$ is bounded in $H^1(G)$;
		\item[(iv)] $\tilde m_{\theta_n}(u_n)\le 1$. 
	\end{enumerate}
\end{theorem}

By Theorem \ref{bt:1} $(ii)$--$(iii)$ and \cite[Remark 1.3]{borthwick2024bounded},
\begin{align*}F_{\tau}'(u_n)+\lambda_n (u_n,\cdot)_{L^2(G)}\to 0\quad \text{in the dual of } H^1(G), \text{ as } n\to +\infty, \end{align*}
where
\begin{equation}\label{b:4}
	\lambda_n =-\frac{1}{\mu}\,\left(F_{\tau}'(u_n)\cdot u_n\right).
\end{equation}
Theorem \ref{bt:1} $(iv)$ gives directly that whenever a subspace
$Y_n\subset T_{u_n}S_{\mu}$ satisfies
\begin{equation}\label{b:5}
	D^2F_{\tau}(u_n)[w,w]=F_{\tau}''(u_n)[w,w]+\lambda_n (w,w)_{L^2(G)}
	<-\theta_n \|w\|^2_{H^1(G)},\qquad \forall\, w\in Y_n\setminus\{0\},
\end{equation}
then necessarily $\dim Y_n \leq 1.$ We refer to $\{\lambda_n\}_{n\in\mathbb{N}}$ from \eqref{b:4} as the sequence of almost Lagrange multipliers.

We shall later see that if $\{u_n\}\subset S_\mu$ converges to some $u\in S_\mu$, then Morse-index information on $u$ as a constrained critical point can be recovered. As a preliminary step, we record properties of the almost Lagrange
multipliers $\{\lambda_n\}\subset \mathbb{R}$.

\begin{lemma}\label{bl:1}
	Let $\{u_n\}\subset S_\mu$, $\{\lambda_n\}\subset \mathbb{R}$, and $\{\theta_n\}\subset \mathbb{R}^+$ satisfy $\theta_n\to 0^+$. Suppose:
	\begin{enumerate}
		\item[(i)] For all sufficiently large $n\in\mathbb{N}$, every subspace $Y_n\subset H^1(G)$ for which
		\begin{equation}\label{b:6}
			F_\tau''(u_n)[\varphi,\varphi]+\lambda_n \|\varphi\|^2_{L^2(G)} < 
			-\theta_n \|\varphi\|^2_{H^1(G)}, \qquad \forall\, \varphi\in Y_n\setminus\{0\},
		\end{equation}
		has dimension at most two, i.e.\ $\dim Y_n \leq 2.$
		\item[(ii)] There exist $\lambda\in\mathbb{R}$, a subspace
		$Y\subset H^1(G)$ with $\dim Y\geq 3$, and $a>0$ such that, for all sufficiently large $n\in\mathbb{N}$,
		\begin{equation}\label{b:7}
			F_\tau''(u_n)[\varphi,\varphi] +\lambda \|\varphi\|^2_{L^2(G)} \leq -a\|\varphi\|^2_{H^1(G)}, \qquad \forall\, \varphi\in Y.
		\end{equation}
	\end{enumerate}
	Then $\lambda_n>\lambda$ for all sufficiently large $n\in\mathbb{N}$. In particular, if \eqref{b:7} holds for every $\lambda<0$, then $\liminf_{n\to\infty}\lambda_n \geq 0.$
\end{lemma}

\begin{proof}
	Assume, for contradiction, that $\lambda_n \leq \lambda$ along a subsequence. By \eqref{b:7}, for every $\varphi \in Y\setminus\{0\}$,
	\begin{align*}
		F_{\tau}''(u_n)[\varphi,\varphi] +\lambda_n \|\varphi\|^2_{L^2(G)}
		&= F_{\tau}''(u_n)[\varphi,\varphi]
		+\lambda \|\varphi\|^2_{L^2(G)} +(\lambda_n-\lambda)\|\varphi\|^2_{L^2(G)} \\
		&\leq -a\|\varphi\|^2_{H^1(G)} +(\lambda_n-\lambda)\|\varphi\|^2_{L^2(G)}.
	\end{align*}
	Since $\lambda_n-\lambda\leq 0$,
	\begin{align*}F_{\tau}''(u_n)[\varphi,\varphi]+\lambda_n \|\varphi\|^2_{L^2(G)} \leq -a\|\varphi\|^2_{H^1(G)}, \qquad \forall\, \varphi\in Y\setminus\{0\}.\end{align*}
	As $\theta_n\to 0^+$, there is $n_0\in\mathbb{N}$ with $\theta_n<a$ for all $n\geq n_0.$ Hence, for $n\geq n_0$ and every $\varphi\in Y\setminus\{0\}$,
	\begin{align*}F_{\tau}''(u_n)[\varphi,\varphi] +\lambda_n \|\varphi\|^2_{L^2(G)} < -\theta_n\|\varphi\|^2_{H^1(G)}.\end{align*}
	So $Y$ satisfies \eqref{b:6}. Since $\dim Y\geq 3$, this contradicts $(i)$, under which any such subspace has dimension at most two. Therefore $\lambda_n>\lambda$ for all sufficiently large $n\in\mathbb{N}$, and the final
	assertion follows.   \qed
\end{proof}

\section{Existence of Mountain-Pass Solutions for the Approximating Problems}\label{bs:2}

Theorem \ref{bt:1} is central to the proof of Theorem \ref{bt:2}. To place the family $E_{W,\tau}(\cdot, G)$ inside its framework, we first show that this family has a mountain pass geometry on $H^1_{\mu}({G})$, uniformly for $\tau\in\left[\frac12,1\right]$.

\begin{lemma}\label{bl:2}
	Let $W$ satisfy assumptions $(W1)$--$(W4)$. Then, for each \(\mu>0\), there exist \(w_1,w_2\in S_\mu\), independent of \(\tau\in\left[\frac12,1\right]\), such that
	\begin{align*}c_{W,\tau} = \inf_{P\in\mathbb{P}}\, \max_{t\in[0,1]} E_{W,\tau}(P(t),{G}) > \max\bigl\{ E_{W,\tau}(w_1,{G}), E_{W,\tau}(w_2,{G})\bigr\}, \qquad \forall\, \tau\in\left[\frac12,1\right],\end{align*}
	where
	\begin{align*}\mathbb{P}= \{ P\in C([0,1],S_{\mu}) :\, P(0)=w_1,\ P(1)=w_2 \}.\end{align*}
\end{lemma}

\begin{proof}
	For any \(\mu,k>0\), put
	\begin{align*}M_{\mu,k} = \left\{ u\in H^1_\mu(G) \;\middle|\;
	\int_{G} |u'(x)|^2\,dx < k \right\},
	\qquad \partial M_{\mu,k} = \left\{ u\in H^1_\mu(G)
	\;:\; \int_{G}|u'(x)|^2\,dx = k \right\}.\end{align*}
	Because \(G\) is noncompact, $M_{\mu,k}\neq \emptyset$ and
	$\partial M_{\mu,k}\neq\emptyset$ for all \(\mu,k>0\). To see this, take \(w\in C_c^\infty(\mathbb R^+)\) with $\|w\|_{L^2(\mathbb R^+)}^2=\mu$, and set $w_t(x)=t^{1/2}w(tx)$, $t>0.$ Then $\|w_t\|_{L^2(\mathbb R^+)}^2=\mu$ and $\|w_t'\|_{L^2(\mathbb R^+)}^2 = t^2\|w'\|_{L^2(\mathbb R^+)}^2$ for all $t>0.$ Since \(G\) has at least one half-line, it contains intervals of arbitrarily large length, so each \(w_t\) can be viewed as an element of \(H^1_\mu(G)\), still written \(w_t\), supported entirely inside \(\ell_1\). In particular, $w_t\in M_{\mu,k}$ for \(t>0\) small. Choosing instead $t=\sqrt{k}/\|w'\|_{L^2(\mathbb R^+)}$ gives $\|w_t'\|_{L^2(\mathbb R^+)}^2=k$, so $w_t\in \partial M_{\mu,k}.$
	
	By the Gagliardo--Nirenberg inequality on metric graphs (Proposition 2.1 in \cite{adami2016threshold}),
	\begin{equation}\label{b:8}
		E_{W,\tau}(u, G) \geq \frac{1}{2}\|u'\|_{L^2(G)}^{2} +\frac{1}{2} \int_G W(x) |u|^2 \,dx - \tau\,\frac{C_p}{p}\, \mu^{\frac{p+2}{4}} \|u'\|_{L^2(G)}^{\frac{p}{2}-1}, \qquad \forall\, u\in H^1_{\mu}(G).
	\end{equation}
	From \eqref{b:8}, for every \(\mu>0\) and every \(u\in \partial M_{\mu,k_0}\), with $k_0 = \frac{1}{2} \left(\frac{p}{2C_p}\right)^{\frac{4}{p-6}} \mu^{-\frac{p+2}{p-6}}$, we obtain
	\begin{align*}
	\inf_{u\in \partial M_{\mu,k_0}} E_{W,\tau}(u,G)
	\geq k_0 \left( \frac{1}{2} - \frac{C_p}{p} \mu^{\frac{p+2}{4}} k_0^{\frac{p-6}{4}} \right) + \frac{1}{2} \int_G W(x) |u|^2 \,dx.
	\end{align*}
	By $(W2)$ we have $\frac{1}{2}\int_G W(x) |u|^2 \, dx \ge 0$, hence
	\begin{equation}\label{b:9}
		\inf_{u\in \partial M_{\mu,k_0}} E_{W,\tau}(u, G)
		\geq k_0 \left( \frac{1}{2} - \frac{C_p}{p} \mu^{\frac{p+2}{4}} k_0^{\frac{p-6}{4}} \right) = \alpha >0,
	\end{equation}
	for every $\tau\in\left[\frac12,1\right]$. On the other hand,
	\begin{align*}	E_{W, \tau}(u,{G}) \leq \frac12\int_{{G}} |u'|^2\,dx 	+\frac12\int_{{G}} W(x)|u|^2\,dx, \end{align*}
	so that, if $u\in M_{\mu,k}$, then
	\begin{align*} 	E_{W,\tau}(u,{G}) \leq 	\frac{k}{2} 	+\frac12\int_{{G}} W(x)|u|^2\,dx. 	\end{align*} 	

	We now construct $w_1$ explicitly, as a bump spreading along the half-line $\ell_1$. Let $v\in C_c^\infty(\mathbb R^+)$ with $\|v\|_{L^2(\mathbb R^+)}^2=\mu$, and for $t>0$ set $v_t(x)=t^{1/2}v(tx)$, regarded as an element of $H^1_\mu(G)$ supported in $\ell_1$. Then$\|v_t\|_{L^2(G)}^2=\mu$ and
	$\|v_t'\|_{L^2(G)}^2=t^2\|v'\|_{L^2(\mathbb R^+)}^2\to 0$ as $t\to0^+$. The support of $v_t$ escapes to infinity along $\ell_1$ as $t\to0^+$, so by assumption $(W3)$,
	\begin{align*}\int_G W(x)|v_t|^2\,dx \;\le\; \Big(\sup_{x\ge a/t} W(x)\Big)\,\mu \;=\; \mu\,\eta(t)\;\xrightarrow[t\to0^+]{}\;0,\end{align*}
	where $\operatorname{supp} v\subset[a,b]\subset(0,\infty)$ and
	$\eta(t):=\sup_{x\ge a/t} W(x)\to0$. Since the nonlinear term is nonnegative,
	\begin{align*}E_{W,\tau}(v_t,G)\;\le\;\frac12\|v_t'\|_{L^2(G)}^2+\frac{\mu}{2}\,\eta(t),
	\qquad\forall\,\tau\in\Big[\tfrac12,1\Big],\end{align*}
	and both terms vanish as $t\to0^+$. So there is $t_1>0$ small enough that, setting $w_1:=v_{t_1}$, one has $\|w_1'\|_{L^2(G)}^2<k_0$ and
	\begin{equation}\label{b:10}
		E_{W,\tau}(w_1,G)\;\le\;\frac{\alpha}{2},\qquad\forall\,\tau\in\Big[\tfrac12,1\Big].
	\end{equation}
	
	For $w_2$, let \(e_1\) be a bounded edge of \( G\), identified with \([-\ell_1/2,\ell_1/2]\). Any compactly supported \(w\in H^1([-\ell_1/2,\ell_1/2])\) with $\|w\|_{L^2(e_1)}^2=\mu$ may be regarded as an element of \(H^1_\mu(G)\) by extension by zero outside \(e_1\). For
	\(t>1\), set $w_t(x)=t^{1/2}w(tx).$ One checks that \(w_t\in H^1_\mu(G)\), the support of \(w_t\) concentrating inside \(e_1\) as \(t\to\infty\). The energy reads
	\begin{align*}
	E_{W,\tau}(w_t, G) = \frac{t^2}{2}\int_{e_1} |w'(x)|^2\,dx
	+ \frac12 \int_{e_1} W\!\left(\frac{x}{t}\right)|w(x)|^2\,dx
	- \frac{\tau\, t^{\frac{p-2}{2}}}{p}\int_{e_1} |w(x)|^p\,dx,
	\end{align*}
	for every $\tau\in\left[\frac12,1\right].$ Since \(p>6\) and, by $(W1)$, $W\in L^\infty$, the right-hand side tends to \(-\infty\) as \(t\to+\infty\). There is therefore \(t_2>0\) large enough that $\|w_t'\|_{L^2( G)}^2 = t^2\|w'\|_{L^2( G)}^2 >2k_0 $ and $E_{W,\tau}(w_t, G)<0$ for all $t>t_2$ and all $\tau\in\left[\frac12,1\right]$. We set $w_2=w_{2t_2}$. By construction,
	\begin{equation}\label{b:11}
		\|w_2'\|_{L^2(G)}^2>2k_0 \qquad\text{and}\qquad
		E_{W,\tau}(w_2, G)<0, \quad \forall\,\tau\in\left[\frac12,1\right].
	\end{equation}
	
	Let \(\mathbb{P}\) and \(c_{W,\tau}\) be as in the statement, for the above \(w_1,w_2\). That \(\mathbb{P}\neq\emptyset\) is clear, since
	\begin{align*}P_0(t) = \frac{\mu^{1/2}\bigl((1-t)w_1+t w_2\bigr)} {\|(1-t)w_1+t w_2\|_{L^2(G)}}, \qquad t\in[0,1],\end{align*}
	lies in \(\mathbb{P}\). By \eqref{b:10} and \eqref{b:11}, for each \(P\in\mathbb{P}\) there is \(t_P\in[0,1]\) with $P(t_P)\in \partial M_{\mu,k_0}$, by continuity. Hence, for every \(\tau\in\left[\frac12,1\right]\) and \(P\in\mathbb{P}\),
	\begin{align*}\max_{t\in[0,1]}E_{W,\tau}(P(t),G) \geq
	E_{W,\tau}(P(t_P),G) \geq \inf_{u\in\partial M_{\mu,k_0}} E_{W,\tau}(u,G) \geq \alpha,\end{align*}
	the last step by \eqref{b:9}. Thus $c_{W,\tau}\geq \alpha.$ Meanwhile, by \eqref{b:10} and \eqref{b:11},
	\begin{align*}\max\bigl\{ E_{W,\tau}(w_1,G), E_{W,\tau}(w_2, G) \bigr\} = E_{W,\tau}(w_1, G) < \frac{\alpha}{2},\end{align*}
	so that
	\begin{align*}c_{W,\tau} >\max\bigl\{ E_{W,\tau}(w_1,G), E_{W,\tau}(w_2, G)
	\bigr\},\end{align*}
	which completes the proof.    \qed
\end{proof}

\begin{rem}
	Assumption $(W3)$ is used decisively in the construction of $w_1$. The bump $v_t(x)=t^{1/2}v(tx)$ has support
	$\{x\in\ell_1: tx\in\operatorname{supp}v\}$, which escapes to infinity along $\ell_1$ as $t\to0^+$, so that
	\begin{align*}\int_G W(x)|v_t|^2\,dx
	\;\le\;\Big(\sup_{x\ge a/t} W(x)\Big)\,\mu\;=\;\mu\,\eta(t)\;\xrightarrow[t\to0^+]{}\;0.\end{align*}
	Were $W$ not to vanish at infinity, say $W(x)\ge W_\infty>0$ for all large $x$ along $\ell_1$, then
	\begin{align*}\int_G W(x)|v_t|^2\,dx\;\xrightarrow[t\to0^+]{}\;W_\infty\,\mu\;\neq\;0,\end{align*}
	the energy $E_{W,\tau}(v_t,G)$ could no longer be pushed below
	$\tfrac{\alpha}{2}$, and the mountain-pass geometry of Lemma \ref{bl:2} would fail. The vanishing of $W$ along the half-lines is thus essential twice over: here, in placing $w_1$ inside the potential well, and in Lemma \ref{bl:3}, where
	the same property controls $\int_G W|\varphi|^2$ for test functions supported far
	out on $\ell_1$.
\end{rem}

\begin{rem}
	The bound $E_{W,\tau}(w_1,G)\le\tfrac{\alpha}{2}$ in \eqref{b:10} is stronger than the mountain-pass argument requires. For the conclusion of Lemma \ref{bl:2} it
	suffices that $E_{W,\tau}(w_1,G)<\alpha$, i.e.\ that $w_1$ lie strictly below the barrier level $\alpha$ given in \eqref{b:9}. Since the construction above yields
	$E_{W,\tau}(v_t,G)\to 0$ as $t\to0^+$ uniformly in $\tau\in[\tfrac12,1]$, and $\alpha>0$, any sufficiently small $t_1>0$ furnishes a function $w_1:=v_{t_1}$
	with $E_{W,\tau}(w_1,G)<\alpha$. The explicit choice $\tfrac{\alpha}{2}$ is retained only for definiteness.
\end{rem}

The next result, together with Lemma \ref{bl:1}, is essential for the convergence of the Palais--Smale sequences produced by Theorem \ref{bt:1}.

\begin{lemma}\label{bl:3}
	Let $W$ satisfy $(W1)$--$(W4)$. Then, for every \(\lambda<0\), there exists a subspace \(Y\subset H^1(G)\) with \(\dim Y=3\) such that
	\begin{equation}\label{b:12}
		\int_{G}|w'(x)|^2\,dx + \int_G W(x)|w(x)|^2 \, dx + \lambda\int_{G}|w(x)|^2\,dx \leq \frac{\lambda}{2}\,\|w\|_{H^1(G)}^2,
		\qquad \forall\, w\in Y.
	\end{equation}
\end{lemma}

\begin{proof}
	Take \(\varphi\in C_0^\infty(\mathbb{R}^+)\) with $\operatorname{supp}\varphi \subset \left[\frac34,\frac54\right]$ and $\int_0^{+\infty} |\varphi(x)|^2\,dx = 1.$ Writing \(\ell_1\) for an arbitrary half-line of \(G\), for \(\rho>0\) define
	\begin{align*}\psi_{0,\rho}(x)=
	\begin{cases}
		0, & x\in G\setminus \ell_1,\\[4pt]
		\rho^{1/2}\varphi(\rho x), & x\in \ell_1,
	\end{cases}
	\quad \psi_{1,\rho}(x)=
	\begin{cases}
		0, & x\in G\setminus \ell_1,\\[4pt]
		\rho^{1/2}\varphi(\rho x-1), & x\in \ell_1,
	\end{cases}\end{align*}
	\begin{align*}\psi_{2,\rho}(x)=
	\begin{cases}
		0, & x\in  G\setminus \ell_1,\\[4pt]
		\rho^{1/2}\varphi(\rho x-2), & x\in \ell_1.
	\end{cases}\end{align*}
	For each \(\rho>0\) we have $\int_{G} |\psi_{j,\rho}|^2\,dx = 1$ for $j=0,1,2$, and \(\psi_{0,\rho},\psi_{1,\rho},\psi_{2,\rho}\) are pairwise orthogonal in \(H^1(G)\), their supports being disjoint. Assume
	\begin{align*}w=\sum_{j=0}^{2}a_j\psi_{j,\rho},
	\qquad a_0,a_1,a_2\in\mathbb R. \end{align*}
	Using the disjointness of the supports, the kinetic and mass terms factor exactly:
	\begin{align*}\int_{G}|w'(x)|^2\,dx = \rho^2\alpha\sum_{j=0}^{2}a_j^2,
	\qquad\int_{G}|w(x)|^2\,dx = \sum_{j=0}^{2}a_j^2,\qquad
	\alpha := \int_0^{+\infty}|\varphi'(x)|^2\,dx > 0,\end{align*}
	and $\|w\|_{H^1(G)}^2 = (\rho^2\alpha+1)\sum_{j=0}^{2}a_j^2$. For the potential term, each $\psi_{j,\rho}$ is supported in
	$\{x\in\ell_1 : \rho x - j \in \operatorname{supp}\varphi\}\subseteq\{x\ge \tfrac{3}{4\rho}\}$, since $\operatorname{supp}\varphi\subset[\tfrac34,\tfrac54]$ and $j\ge0$. Setting
	\begin{align*}\eta(\rho):=\sup_{x\ge \frac{3}{4\rho}} W(x),\end{align*}
	assumption $(W3)$ gives $\eta(\rho)\to0$ as $\rho\to0^+$, and by disjoint supports, we obtain
	\begin{align*}\int_G W(x)|w(x)|^2\,dx
	= \sum_{j=0}^{2}a_j^2\int_G W(x)|\psi_{j,\rho}(x)|^2\,dx
	\le \eta(\rho)\sum_{j=0}^{2}a_j^2\int_G|\psi_{j,\rho}(x)|^2\,dx
	= \eta(\rho)\sum_{j=0}^{2}a_j^2.\end{align*}
	Combining the three estimates,
	\begin{align*}\int_{G}|w'|^2\,dx + \int_G W(x)|w|^2\,dx + \lambda\int_{G}|w|^2\,dx \le \bigl(\rho^2\alpha + \eta(\rho) + \lambda\bigr)\sum_{j=0}^{2}a_j^2.\end{align*}
	Therefore, for every nontrivial $w\in\operatorname{span}\{\psi_{0,\rho},\psi_{1,\rho},\psi_{2,\rho}\}$,
	\begin{align*}\frac{\displaystyle\int_{G}|w'|^2\,dx + \int_G W(x)|w|^2\,dx + \lambda\int_{G}|w|^2\,dx}
	{\|w\|_{H^1(G)}^2} \le \frac{\rho^2\alpha+\eta(\rho)+\lambda}{\rho^2\alpha+1} \le \frac{\lambda}{2},\end{align*}
	for every sufficiently small $\rho>0$. So \eqref{b:12} holds on \(\operatorname{span}\{\psi_{0,\rho},\psi_{1,\rho},\psi_{2,\rho}\}\). Setting
	$Y=\operatorname{span}\{\psi_{0,\rho},\psi_{1,\rho},\psi_{2,\rho}\}$, we have
	\(\dim Y=3\), completing the proof.   \qed
\end{proof}

	\begin{lemma}\label{bl:7}
	Let $\ell\simeq[0,+\infty)$ be an unbounded edge of $G$ and let $W$ satisfy $(W1)$--$(W4)$. Consider the zero-energy equation on $\ell$,
	\begin{equation}\label{b:57}
		-\varphi''+W(x)\varphi=0, \qquad x\in(0,+\infty).
	\end{equation}
	\begin{enumerate}
		\item[(i)] Every solution of \eqref{b:57} lying in $L^2(\ell)$ is nonincreasing and of one sign near $+\infty$, and satisfies
		\begin{equation}\label{b:58}
		\varphi(x)=\int_x^{+\infty}(t-x)\,W(t)\,\varphi(t)\,dt,
		\qquad \forall\, x \ \text{large}.
		\end{equation}
		\item[(ii)] If $\int_1^{+\infty}xW(x)\,dx<+\infty$, then the only solution of \eqref{b:57} in $L^2(\ell)$ is $\varphi\equiv0$. We then call $\ell$ a \emph{regular} edge.
		\end{enumerate}
	\end{lemma}
	\begin{proof}
	$(i)$ Let $\varphi\in L^2(\ell)$ solve \eqref{b:57}, $\varphi\not\equiv0$. If $\varphi(a)=\varphi(b)=0$ with $0\le a<b$, then on $(a,b)$ we may assume $\varphi>0$, so $\varphi''=W\varphi\ge0$ by $(W2)$ and $\varphi$ is convex on
	$[a,b]$ with vanishing endpoint values, whence $\varphi\le0$ on $[a,b]$, a contradiction. So $\varphi$ has at most one zero and, replacing $\varphi$ by $-\varphi$ if necessary, there is $x_0\ge0$ with $\varphi>0$ on $(x_0,+\infty)$. On $(x_0,+\infty)$ we have $\varphi''\ge 0$, so $\varphi'$ is
	nondecreasing and $\varphi'(x)\to L\in(-\infty,+\infty]$. If $L>0$ then $\varphi(x)\to+\infty$, contradicting $\varphi\in L^2(\ell)$; hence $L\le0$ and therefore $\varphi'\le0$ on $(x_0,+\infty)$. Thus $\varphi$ is positive and nonincreasing, so $\varphi(x)\to c\ge0$, and $\varphi\in L^2(\ell)$ forces $c=0$; if $L<0$ then $\varphi(x)\to-\infty$, so $L=0$. Integrating $\varphi''=W\varphi$ over $(x,R)$ and letting $R\to+\infty$ gives $-\varphi'(x)=\int_x^{+\infty}W\varphi\,dt\ \ (\ge0)$; integrating once more and using Fubini yields \eqref{b:58}.\\
    $(ii)$ Suppose $\varphi\not\equiv0$ and keep the notation of $(i)$. Since $\varphi$ is nonincreasing on $(x_0,+\infty)$, \eqref{b:58} gives, for $x>x_0$,
		\begin{align*}
			\varphi(x)=\int_x^{+\infty}(t-x)W(t)\varphi(t)\,dt
			\;\le\;\varphi(x)\int_x^{+\infty}(t-x)W(t)\,dt
			\;\le\;\varphi(x)\int_x^{+\infty}t\,W(t)\,dt .
		\end{align*}
		By hypothesis $\int_x^{+\infty}tW(t)\,dt\to0$ as $x\to+\infty$, so there is $x_1>x_0$ with $\int_{x_1}^{+\infty}tW(t)\,dt<1$, and the displayed inequality forces $\varphi(x_1)\le\theta\varphi(x_1)$ with $\theta = \int_{x_1}^{+\infty}tW(t)\,dt<1$, i.e. $\varphi(x_1)\le0$, contradicting $\varphi>0$ on $(x_0,+\infty)$. Hence $\varphi\equiv0$.   \qed
	\end{proof}
	\begin{rem}
		The decay required in Lemma \ref{bl:7}$(ii)$ is not merely technical, and it is not implied by $(W1)$--$(W3)$ together with integrability of $W$. Taking
		$W(x)=2(1+x)^{-2}$ on $\ell$, one has $W\in L^\infty(\ell)\cap L^1(\ell)$,
		$W\ge0$ and $xW(x)\to0$ as $x\to+\infty$, so that $(W1)$--$(W3)$ hold; yet $\varphi(x)=(1+x)^{-1}$ satisfies $\varphi''=2(1+x)^{-3}=W\varphi$ with $\varphi>0$ and $\varphi\in H^1(\ell)$. Thus \eqref{b:57} admits a positive $H^1$ solution on $\ell$, and no conclusion on the sign of the multiplier can be drawn from $(W1)$--$(W3)$ alone. Here $\int_1^{+\infty}xW(x)\,dx=+\infty$, so $\ell$ is not regular; the inverse-square decay is exactly the borderline, since $\varphi(x)=(1+x)^{-a}$ with $a>\tfrac12$ solves \eqref{b:57} for
		$W(x)=a(a+1)(1+x)^{-2}$.
	\end{rem}

We close this section with a definition; throughout, \(\chi_A\) denotes the characteristic function of \(A\).

\begin{definition}
	Let \( G\) be a graph, \( G'\subset G\) a subgraph, and $U\in C(G)\cap H^1_{\mathrm{loc}}(G)$, not necessarily in \(H^1(G)\), a solution of
	\begin{align*}
	\begin{cases}
	-U'' + W(x){U} +\lambda U = \tau |U|^{p-2}U\,\chi_{G'} & \text{in } G,\\[4pt]
	\displaystyle\sum_{e\succ v} U'(v)=0,
	& \text{for every vertex } v \in V,
	\end{cases}
	\end{align*}
	with \(\lambda,\tau\in\mathbb R\).  Associated with \(U\), define the quadratic form
	\begin{equation}\label{b:14}
	Q_W(\varphi;U,G) = \int_{G} \Bigl( |\varphi'(x)|^2
	+ \bigl(W(x)+\lambda-(p-1)\tau |U(x)|^{p-2}\chi_{G'}(x)\bigr)
	\varphi^2(x) \Bigr)\,dx,
	\end{equation}
	for every $\varphi\in H^1(G)\cap C_c(G).$
\end{definition}

The Morse index \(m(U)\) of \(U\) is the largest dimension of a subspace $Y \subset H^1(G)\cap C_c(G)$ with $Q_W(\varphi; U, G)<0$ for all $\varphi\in Y\setminus\{0\}.$ When $W\equiv0$ we write $Q$ for the corresponding $W$-free form.

 The main goal of this section is the following.

\begin{prop}\label{bp:1}
	Let \(\mu>0\) be fixed and $W$ satisfy $(W1)$--$(W4)$. Then, for almost every \(\tau\in I_\tau\), there is a pair $(u_\tau,\lambda_\tau)\in H^1_\mu(G)\times \mathbb R^+$ solving
	\begin{equation}\label{b:15}
		\begin{cases}
			-u_\tau''+ W(x)u_\tau+\lambda_\tau u_\tau = \tau\,\chi(x)\,u_\tau^{\,p-1},
			& \text{in }  G,\\[4pt]
			u_\tau>0, & \text{in }G,\\[4pt]
			\displaystyle\sum_{e\succ v}u_{\tau}'(v)=0,
			& \text{for every vertex } v \in V.
		\end{cases}
	\end{equation}
	Moreover, $E_{W,\tau}(u_\tau, G)=c_{W, \tau}$ and $m(u_\tau)\leq 2.$
\end{prop}

\begin{proof} 
	We apply Theorem \ref{bt:1} to the family $E_{W,\tau}$, with \(\mathbb{P}\) as in Lemma \ref{bl:2}. Set
	\begin{align*}A(u) = \frac{1}{2} \int_G |u'|^2 \,dx + \frac{1}{2} \int_G W(x)|u|^2 \, dx,  \quad B(u) = \frac{\tau}{p}\int_G |u|^p \, dx.\end{align*}
	By $(W2)$, assumption \eqref{b:1} holds. Writing $E'_{W, \tau}$ and $E''_{W, \tau}$ for the first and second derivatives of $E_{W, \tau}$, both are of class $C^1$ and hence locally H\"older continuous (Definition \ref{bd:1}) on $H^1_\mu(G)$. The remaining hypotheses hold by Lemma \ref{bl:2}. Thus, for almost every $\tau \in [1/2, 1]$, there exist a bounded sequence $\{u_{n,\tau}\} \subset H^1_\mu(G)$, denoted $\{u_n\}$, and $\{\theta_n\} \subset \mathbb{R}^+$ with $\theta_n \to 0^+$ such that
	\begin{equation}\label{b:16}
		E'_{W,\tau}(u_n) + \lambda_n (u_{n}, \cdot) \to 0 \text{ in the dual of } H^1_\mu(G) \text{ as } n \to \infty,
	\end{equation}
	where
	\begin{equation}\label{b:17}
		\lambda_n =-\frac{1}{\mu}\,E_{W, \tau}'(u_n)[u_n].
	\end{equation}
	Moreover, if
	\begin{equation}\label{b:18}
		\int_{G} \Bigl( |\varphi'|^2
		+ \bigl(W(x)+\lambda_n-(p-1)\tau |u_n|^{p-2}\chi(x)\bigr) \varphi^2 \Bigr)\,dx
		= E''_{W,\tau}(u_n)[\varphi, \varphi] + \lambda_n \|\varphi\|^2_{L^2(G)}
		< -\theta_n \|\varphi\|^2_{H^1(G)}
	\end{equation}
	holds for every $\varphi \in Y_n \setminus \{0\}$ in a subspace $Y_n$ of $T_{u_n}S_\mu$, then $\dim Y_n \le 1$. Since $u \in H^1_\mu(G)$ implies $|u| \in H^1_\mu(G)$, since $w_1, w_2 \ge 0$, since the map $u \mapsto |u|$ is continuous, and since $E_{W,\tau}(|u|) = E_{W,\tau}(u)$, we may take $\{u_n\}$ with $u_n \ge 0$ on $G$; see \cite[Remark 1.8]{borthwick2024bounded}.
	
	By \eqref{b:17}, boundedness of $\lambda_n$ implies boundedness of $\{u_n\}$. Passing to a subsequence, $\lim_{n} \lambda_n = \lambda_\tau$ for some $\lambda_\tau \in \mathbb{R}$, and there is $u_\tau \in H^1(G)$ with
	\begin{align}
		u_n &\rightharpoonup u_\tau \quad \text{ in } H^1(G), \label{b:19}\\
		u_n &\to u_\tau \quad \text{ in } L^p_{\mathrm{loc}}(G), \ p>2, \label{b:20}\\
		u_n(x) &\to u_\tau(x) \quad \text{for a.e. } x\in G, \label{b:21}
	\end{align}
	whence $u_\tau \ge 0$. Using \eqref{b:16}, $(W1)$ and $\lambda_n \to \lambda_\tau$, we get
	\begin{equation}\label{b:22}
		\int_G u'_n \varphi' \, dx + \int_G W(x) u_n \varphi\, dx  - \tau \int_G \chi(x)|u_n|^{p-2}u_n \varphi \, dx + \lambda_\tau \int_G u_n \varphi \, dx = o(1)\|\varphi\|_{H^1(G)}, 
	\end{equation}
	for every $\varphi \in H^1(G)$. By \eqref{b:19}--\eqref{b:21}, $u_\tau$ satisfies
	\begin{equation}\label{b:23}
		\begin{cases}
			-u_\tau''+ W(x)u_\tau+\lambda_\tau u_\tau =\tau\,\chi(x)\,u_\tau^{p-1}
			& \text{on } G,\\[2mm]
			\displaystyle\sum_{e\succ v} u'_{\tau}(v)=0
			& \text{for every vertex } v\in V.
		\end{cases}
	\end{equation}
	To prove $u_\tau \not\equiv 0$, we first show $\lambda_\tau\ge 0$. Since $T_{u_n}S_\mu$ has codimension $1$, if \eqref{b:18} holds for all $\varphi \in Y_n \setminus \{0\}$ in a subspace $Y_n \subset H^1(G)$, then $\dim Y_n \le 2$. By Lemma \ref{bl:3}, for every $\lambda<0$ there exist a subspace $Y \subset H^1(G)$ with $\dim Y \ge 3$ and $a>0$ such that, for $n\in\mathbb{N}$ large,
	\begin{align*}
	E''_{W,\tau}{(u_n) [\varphi, \varphi]} + \lambda \|\varphi\|^2_{L^2(G)} &\le \int_{ G}|\varphi'(x)|^2\,dx + \int_G W(x)|\varphi(x)|^2 \, dx + \lambda\int_{G}|\varphi(x)|^2\,dx \\&\leq -a\|{\varphi}\|_{H^1(G)}^2,
	\end{align*}
	for all $\varphi \in Y\setminus\{0\}$. Lemma \ref{bl:1} then gives $\lambda_\tau \ge 0$.
	
	Furthermore, by \eqref{b:23},
	\begin{equation}\label{b3:24}
		\int_G u'_\tau \varphi' \, dx + \int_G W(x) u_\tau\varphi\, dx  - \tau \int_G \chi(x)|u_\tau|^{p-2}u_\tau \varphi \, dx + \lambda_\tau \int_G u_\tau \varphi \, dx =0.
	\end{equation}
	Subtracting \eqref{b:22} from \eqref{b3:24} and using \eqref{b:19}--\eqref{b:21},
	\begin{equation}\label{b:24}
		\int _G |(u_n- u_\tau)'|^2 \, dx + \int_G W(x)|u_n - u_\tau|^2 \, dx + \lambda_\tau \int_G |u_n - u_\tau|^2 \, dx \to 0   \quad\text{ as } n\to \infty.
	\end{equation}
	If $u_\tau \equiv 0$, then
	\begin{equation}\label{b:25}
		\int_G |u'_n|^2 \, dx+ \int_G W(x)|u_n|^2 \, dx + \lambda_\tau \int_G |u_n|^2 \, dx \to 0.
	\end{equation}
	If \(\lambda_\tau>0\), then \((W1)\) gives
	\begin{align*}\int_G |u_n'|^2\,dx +\int_G W(x)|u_n|^2\,dx
	+\lambda_\tau\int_G |u_n|^2\,dx \geq
	\lambda_\tau\int_G |u_n|^2\,dx
	= \lambda_\tau \mu >0,\end{align*}
	so the left side stays away from $0$, contradicting \eqref{b:25}. Thus, if \(u_\tau\equiv 0\), necessarily \(\lambda_\tau=0\). If \(\lambda_\tau=0\), then \eqref{b:25} yields
	\begin{align*}
	\int_G |u_n'|^2\,dx+\int_G W(x)|u_n|^2\,dx \to 0,
	\end{align*}
	and by \((W1)\), $\|u_n'\|_{L^2(G)}\to 0$. By Gagliardo--Nirenberg,
	\begin{align*}
	\|u_n\|_{L^p(G)}^p \leq
	C\,\|u_n\|_{L^2(G)}^{\frac{p+2}{2}}
	\|u_n'\|_{L^2(G)}^{\frac{p-2}{2}}
	\to 0,
	\end{align*}
	and therefore
	\begin{align*}
	E_{W,\tau}(u_n) = \frac12\int_G |u_n'|^2\,dx +\frac12\int_G W(x)|u_n|^2\,dx -\frac{\tau}{p}\int_G \chi(x)|u_n|^p\,dx \to 0.
	\end{align*}
	Since \(E_{W,\tau}(u_n)\to c_{W,\tau}\), we get $c_{W,\tau}=0$, contradicting \(c_{W,\tau}>0\). Hence \(u_\tau \not\equiv 0\).
	
	We now claim \(u_\tau>0\) on \(G\). Suppose $u_\tau(x_0)=0$ for some \(x_0\in G\). If \(x_0\) is interior to an edge, it is a minimum of \(u_\tau\), so $u_\tau'(x_0)=0$. If \(x_0\) is a vertex, nonnegativity of \(u_\tau\) and the Kirchhoff condition force all outgoing derivatives to vanish, i.e. $u_\tau'(x_0)=0$. Since \(u_\tau\) solves \eqref{b:23} on each edge, ODE uniqueness gives $u_\tau\equiv 0$ on any edge through \(x_0\); iterating across the finitely many vertices and edges yields $u_\tau\equiv 0$ on $G$, contradicting $u_\tau\not\equiv 0$. Hence $u_\tau>0$ on $G$.

    It remains to upgrade $\lambda_\tau\ge0$ to $\lambda_\tau>0$. Let \(\ell\) be an unbounded edge, identified with \([0,+\infty)\). Since $\chi\equiv 0$ on \(\ell\), the restriction of \(u_\tau\) to \(\ell\) satisfies
		\begin{align*}
			-u_\tau'' + W(x)\,u_\tau + \lambda_\tau\, u_\tau = 0
			\qquad \text{on } (0,+\infty),
		\end{align*}
	with \(u_\tau>0\) and \(u_\tau\in H^1(\ell)\). Choose for $\ell$ the edge given by $(W4)$, which is regular in
	the sense of Lemma \ref{bl:7}$(ii)$. If $\lambda_\tau=0$, then $u_\tau|_\ell$ is a solution of \eqref{b:57} belonging to $L^2(\ell)$, and Lemma \ref{bl:7}$(ii)$ forces $u_\tau\equiv0$ on $\ell$, contradicting $u_\tau>0$ on $G$. Together with $\lambda_\tau\ge0$ this gives $\lambda_\tau>0$.

	As \(\lambda_\tau>0\), \eqref{b:24} and $(W2)$ give
	\begin{align*}
	\int_G |(u_n-u_\tau)'|^2\,dx
	+ \int_G W(x)|u_n - u_\tau|^2 \, dx +
	\lambda_\tau\int_G |u_n-u_\tau|^2\,dx
	\to 0,
	\end{align*}
	and all terms being nonnegative, $u_n\to u_\tau$ strongly in $H^1(G)$.
	
	Finally, we prove $m(u_\tau)\le 2$. Since \(T_{u_\tau}H^1_\mu(G)\) has codimension \(1\) in \(H^1(G)\), it suffices to show the constrained Morse index of \(u_\tau\) on \(H^1_\mu(G)\) is at most \(1\). Suppose not. By Definition \ref{bd:3}, there is a two-dimensional subspace $Y_0 \subset T_{u_\tau}H^1_\mu(G)$ with
	\begin{equation}\label{b:26}
		D^2E_{W,\tau}(u_\tau)[w,w]<0, \quad \forall\, w\in Y_0\setminus\{0\}.
	\end{equation}
	As \(Y_0\) is finite-dimensional, $S_{Y_0}=\{w\in Y_0:\|w\|_{H^1(G)}=1\}$ is compact, and $w\mapsto D^2E_{W,\tau}(u_\tau)[w,w]$ is continuous and strictly negative there, so there is \(\beta>0\) with
	\begin{align*}
	D^2E_{W,\tau}(u_\tau)[w,w]< -\beta, \quad \forall\, w\in S_{Y_0}.
	\end{align*}
	For \(w\in Y_0\setminus\{0\}\), putting $\tilde w=w/\|w\|_{H^1(G)}$ and using homogeneity of \(D^2E_{W,\tau}(u_\tau)\),
	\begin{align*}
	D^2E_{W,\tau}(u_\tau)[w,w] = \|w\|_{H^1(G)}^2
	D^2E_{W,\tau}(u_\tau)[\tilde w,\tilde w]
	< -\beta \|w\|_{H^1(G)}^2.
	\end{align*}
	Since \({E_{W,\tau}'}\) and \({E_{W,\tau}''}\) are \(\alpha\)-H\"older continuous on bounded subsets of \(H^1(G)\) for some \(\alpha\in(0,1]\), by \cite[Corollary 1]{borthwick2024bounded} there is \(\eta_1>0\) small such that, for every \(v\in S_\mu\) with $\|v-u_\tau\|_{H^1(G)}\leq \eta_1$,
	\begin{equation}\label{b:28}
		D^2{E_{W,\tau}}(v)[w,w] < -\frac{\beta}{2}\|w\|_{H^1(G)}^2, 
		\qquad \forall\, w\in Y_0\setminus\{0\}.
	\end{equation}
	Since \(u_n\to u_\tau\) strongly in \(H^1(G)\), we have $\|u_n-u_\tau\|_{H^1(G)}\leq \eta_1$ for large \(n\). Applying \eqref{b:26} and \eqref{b:28} with \(v=u_n\), and using \(\theta_n\to0^+\),
	\begin{equation}\label{b:29}
		D^2{E_{W,\tau}}(u_n)[w,w] < -\frac{\beta}{2}\|w\|_{H^1(G)}^2
		< -\theta_n\|w\|_{H^1(G)}^2, \qquad \forall\, w\in Y_0\setminus\{0\},
	\end{equation}
	for all large \(n\). Since \(\dim {Y_0}=2 >1\) and
	\begin{align*}
	{E_{W,\tau}''}(u_n)[w,w]+\lambda_\tau\|w\|_{L^2(G)}^2
	= D^2{E_{W,\tau}}(u_n)[w,w],
	\end{align*}
	it follows from \eqref{b:29} that
	\begin{align*}
	{E_{W,\tau}''}(u_n)[w,w] + \lambda_\tau\|w\|_{L^2(G)}^2
	< -\theta_n\|w\|_{H^1(G)}^2
	\qquad \forall\, w\in Y_0\setminus\{0\}.
	\end{align*}
	As \(\dim Y_0>1\), this contradicts \eqref{b:18}. Hence $\tilde m_0(u_\tau)\leq 1$, i.e. the constrained Morse index of \(u_\tau\) is at most one. Since \(T_{u_\tau}S_\mu\) has codimension one in \(H^1(G)\), we conclude $m(u_\tau)\leq 2.$
    \qed

\end{proof}

\section{A blow-up analysis and proof of the main theorem}\label{bs:3}

By Proposition \ref{bp:1}, there is a sequence $\tau_n \to 1^-$ together with a corresponding sequence $u_{\tau_n}\in H^1_\mu(G)$ of constrained critical points of $E_{W,\tau_n}$ on $H^1_\mu(G)$. Each $u_{\tau_n}$ is produced at the mountain-pass level $c_{W,\tau_n}$, that is ${E_{W, \tau_n}}(u_{\tau_n})=c_{W,\tau_n}$, and satisfies $m(u_{\tau_n})\leq 2$.
The associated Lagrange multipliers are strictly positive, $\lambda_n>0$ for all $n\in\mathbb N.$

To obtain a solution of the original problem at a positive energy level, it suffices to show that $\{u_{\tau_n}\}$ converges; the key step is thus the boundedness of $\{u_{\tau_n}\}$ in $H^1(G)$.

By Lemma {\ref{bl:2}} and the monotonicity of the mountain-pass levels
$c_{W,\tau_n}$,
\begin{align*}
E_{W,1}(w_1,G) \leq E_{W,\tau}(w_1,G) \leq c_{W,\tau_n}  \leq c_{W,1/2},
\qquad \forall\, \tau\in\left[\tfrac12,1\right],
\end{align*}
so the levels $c_{W,\tau_n}$ are uniformly bounded. Using the Kirchhoff conditions,
\begin{align*}
\int_G |u_{\tau_n}'|^2\,dx + \int_G W(x)|u_{\tau_n}|^2\,dx
+ \lambda_n \int_G |u_{\tau_n}|^2\,dx = {\tau_n}\int_\mathcal{K} |u_{\tau_n}|^p\,dx,
\end{align*}
whence
\begin{equation}\label{b:30}
	{\tau_n}\int_\mathcal{K} |u_{\tau_n}|^p\,dx  = \int_G |u_{\tau_n}'|^2\,dx
	+ \int_G W(x)|u_{\tau_n}|^2\,dx + \lambda_n\mu.
\end{equation}
On the other hand, from
\begin{align*}
E_{W,{\tau_n}}(u_{\tau_n},G) =
\frac12\int_G |u_{\tau_n}'|^2\,dx
+ \frac12\int_G W(x)|u_{\tau_n}|^2\,dx
- \frac{{\tau_n}}{p}\int_\mathcal{K}|u_{\tau_n}|^p\,dx = c_{W,\tau_n},
\end{align*}
substituting \eqref{b:30} into $E_{W,{\tau_n}}(u_{\tau_n},G)$ gives
\begin{align*}
\left(\frac12-\frac1p\right)
\left( \int_G |u_{\tau_n}'|^2\,dx +
\int_G W(x)|u_{\tau_n}|^2\,dx \right)
= c_{W,\tau_n}+\frac{\lambda_n\mu}{p}.
\end{align*}
By $(W2)$,
\begin{align*}
\left(\frac12-\frac1p\right)
\int_G |u_{\tau_n}'|^2\,dx \leq
c_{W,\tau_n}+\frac{\lambda_n\mu}{p}.
\end{align*}
So if $\{\lambda_n\}$ is bounded, then $\{u_{\tau_n}\}$ is bounded in $H^1(G)$. We therefore assume, up to a subsequence, that $\lambda_n \to +\infty$, and derive a contradiction via a blow-up analysis of $\{u_{\tau_n}\}$ as in
\cite{chang2023normalized, esposito2011pointwise, carrillo2026blow}.

We stress that this section concerns the regime $\lambda_n\to+\infty$. The obstruction specific to the present setting occurs as $\lambda\to0$, on the half-lines where $W\not\equiv0$ and no explicit solution is available; it is removed by $(W4)$ and Lemma \ref{bl:7}.

Writing $u_n =u_{\tau_n}$ with $\tau_n\to 1^{-}$, we study a sequence of positive solutions $\{u_n\}\subset H^{1}(G)$ of
\begin{equation}\label{b:31}
	\begin{cases}
		-u_n''+W(x)u_n+\lambda_nu_n =\tau_n\chi(x)u_n^{p-1},
		& \text{on } G,\\[2mm]
		u_n>0, & \text{on } G,\\[2mm]
		\displaystyle\sum_{e\succ v}u'_{e,n}(v)=0,
		& \forall\, v\in V,
	\end{cases}
\end{equation}
where $m(u_n)\leq 2$ for every $n\in\mathbb N$ and $\lambda_n\to+\infty.$

We first observe that the local maxima of $u_n$ lie on the compact core $\mathcal{K}$.

\begin{lemma}\label{bl:4}
	Let $W$ satisfy $(W1)$--{$(W2)$} and let $x_n \in G$ be a local maximum point of $u_n$. Then $x_n \in \mathcal{K}$ and
	\begin{equation}\label{b:32}
		u_n(x_n) \ge \left({W(x_n)+ \lambda_n}\right)^{\frac{1}{p-2}}.
	\end{equation}
\end{lemma}
\begin{proof}
Suppose $x_n$ lies in the interior of a half-line $\ell_i\in E$, identified with $[0,+\infty)$, so $x_n\in(0,+\infty)$; by elliptic regularity $u_n|_{\ell_i}\in C^2([0,+\infty))$. As $x_n$ is a local maximum, $u_n''(x_n)\leq 0$.Since $\chi\equiv0$ on $\ell_i$, $u_n$ satisfies on $\ell_i$
\begin{align*}
	-u_n''+W(x)u_n+\lambda_nu_n=0,
	\end{align*}
	so that
	\begin{align*}
	u_n''(x_n) = \bigl(W(x_n)+\lambda_n\bigr)u_n(x_n) \le 0.
	\end{align*}
	Using $u_n(x_n)>0$, $(W2)$ and $\lambda_n>0$ gives $u_n''(x_n)>0$, contradicting
	$u_n''(x_n)\leq 0$. Hence $x_n\in \mathcal{K}$. Since $u_n''(x_n)\leq 0$,
	\eqref{b:31} yields
	\begin{align*}
	\bigl(W(x_n)+\lambda_n\bigr)u_n(x_n)\leq \tau_n\chi(x_n)u_n(x_n)^{p-1},
	\end{align*}
	and, using $\chi(x_n)=1$ and $\tau_n\le 1$, we obtain \eqref{b:32},
	\begin{align*}
	u_n(x_n)\geq \left({\lambda_n+W(x_n)}\right)^{\frac{1}{p-2}}.
	\end{align*} \qed
\end{proof}

For $x_0\in G$ and $r>0$, write $B_r(x_0)=\{x\in G:\operatorname{dist}(x,x_0)<r\}$. For $m\geq 1$, let $G_m$ be the star graph of $m$ half-lines joined at a common
vertex identified with $0$; in particular $G_1=\mathbb{R}^+$, while $G_2$ is isometric to $\mathbb{R}$. Since $m(u_n)\leq 2$, the local structure of $G$ near a maximum of $u_n$ is simple enough to describe the asymptotics of $\{u_n\}$ as
$\lambda_n\to+\infty$.

\begin{lemma}\label{bl:5}
	Let $U\in H^1_{\mathrm{loc}}(G_m)$ solve
	\begin{equation}\label{b4:32}
		\begin{cases}
			- U'' + \lambda U =
			\tau U^{p-1}\chi_{\ell_1\cup\cdots\cup\ell_j}
			& \text{in } G_m,\\[2mm]
			U>0
			& \text{in } G_m,\\[2mm]
			\displaystyle \sum_{i=1}^{m} U_i'(0)=0,
		\end{cases}
	\end{equation}
    	for some $p>2$, $\tau,\lambda>0$, where $U_i$ is the restriction of $U$ to the $i$-th half-line $\ell_i$ and {$1\le j\le m$}. Assume $U$ is bounded and stable outside a compact set $\mathcal{C}\subset G_m$, i.e.\
	\begin{align*}
	{Q}(\varphi;U,G_m)\ge 0, \quad \forall\,
	\varphi\in H^1(G_m)\cap C_c(G_m\setminus \mathcal{C}).
	\end{align*}
	Then $U(x)\to 0$ as $\operatorname{dist}(x,0)\to+\infty$, and $U\in H^1(G_m)$.
\end{lemma}

\begin{proof}
	On $\ell_1,\ldots,\ell_j$ the argument is as in \cite{chang2023normalized}, based on \cite[Theorem 2.3]{esposito2011pointwise}. On $\ell_{j+1},\ldots,\ell_m$, $U$ solves $U''-\lambda U =0$ for some $\lambda>0$,
	and, $U$ being bounded, its restriction to each such half-line is $U(x)=c\,e^{-\sqrt{\lambda}\,x}$ for some $c>0$.  \qed
\end{proof}

\begin{lemma}\label{bl:6}
	Let $U\in H^{1}(G_m)$ be a nontrivial solution of \eqref{b4:32}. Then $m(U)>0$.
\end{lemma}
\begin{proof}
By the Kirchhoff condition and \eqref{b4:32}, testing against $U$,
	\begin{align*}
	\int_{G_m} \left( |U'|^2 + \lambda U^2
	\right)\,dx = \tau \sum_{i=1}^{j}
	\int_{\ell_i} |U|^p\,dx.
	\end{align*}
	Therefore, with the $W$-free form $Q$,
	\begin{align*}
		{Q}(U;U,G_m) &= \int_{G_m} \left( |U'|^2 + \lambda U^2
		\right)\,dx - \tau (p-1) \sum_{i=1}^{j}
		\int_{\ell_i} |U|^p\,dx \\
		&= \tau(2-p) \sum_{i=1}^{j}
		\int_{\ell_i} |U|^p\,dx <0.
	\end{align*}
    	Since $p>2$, $U\not\equiv 0$ and $j\ge1$, the form has a negative direction; by density of $H^1(G_m)\cap C_c(G_m)$ in $H^1(G_m)$ there is $\varphi\in H^1(G_m)\cap C_c(G_m)$ with $Q(\varphi;U,G_m)<0$.   \qed
\end{proof}

Throughout, $\chi_A$ denotes the characteristic function of $A\subset G$.

\begin{prop}\label{bp:2}
	Let $G$ be a noncompact metric graph with finitely many edges and vertices, $W$ satisfy $(W1)$--{$(W2)$}, and $p>2$. Let $\{u_n\}\subset H^1(G)$ be positive solutions of \eqref{b:31} for some $\lambda_n\in\mathbb{R}$, $\tau_n\in(0,1]$ with $\tau_n\to1^-$, with $\lambda_n\to+\infty$ and $m(u_n)\leq k^*$ for some $k^*\geq 1$.  Let $x_n\in G$ satisfy, for some $R_n\to+\infty$,
    
	\begin{align*}
	u_n(x_n) = \max_{B_{R_n\tilde{\varepsilon}_n}(x_n)}u_n,
	\qquad \tilde{\varepsilon}_n = u_n(x_n)^{-\frac{p-2}{2}} \to 0,
	\end{align*}
	and assume
	\begin{equation}\label{b:33}
		\limsup_{n\to\infty} \frac{\operatorname{dist}(x_n,V)} {\tilde{\varepsilon}_n} = +\infty.
	\end{equation}
	Then, up to a subsequence:
	\begin{enumerate}
		\item[(i)] All $x_n$ lie in the interior of the same edge $e$, identified with $I_e=[0,\ell_e]$ if $e$ is bounded, or $I_e=[0,\infty)$ if $e$ is a half-line.
		\item[(ii)] With $\varepsilon_n=\lambda_n^{-1/2}$,
		\begin{equation}\label{b:34}
			\frac{\tilde{\varepsilon}_n}{\varepsilon_n}
			\to  c \in (0,1], \qquad
			\frac{\operatorname{dist}(x_n,V)}{\varepsilon_n}
			\to +\infty \quad (n\to\infty),
		\end{equation}
		and the rescaled sequence
		\begin{equation}\label{b:35}
			v_n(y) = \varepsilon_n^{\frac{2}{p-2}}
			u_n(x_n+\varepsilon_n y), \qquad
			y\in \left[ -\tfrac{x_n}{\varepsilon_n},
			\tfrac{\ell_e-x_n}{\varepsilon_n} \right],
		\end{equation}
		converges in $C^2_{\mathrm{loc}}(\mathbb{R})$ to $U\in H^1(\mathbb{R})$, the unique positive finite-energy solution of
		\begin{align*}
		\begin{cases}
			-U''+U = U^{\,p-1}, & \text{in }\mathbb{R},\\[2mm]
			U>0, & \text{in }\mathbb{R},\\[2mm]
			U(0) = \displaystyle\max_{\mathbb{R}}{U},\\[2mm]
			U(x)\to 0, & \text{as } \operatorname{dist}(x,0)\to+\infty.
		\end{cases}
		\end{align*}
		\item[(iii)] There is $\{\varphi_n\}\subset H^1(G)\cap C_c(G)$ with $\operatorname{supp}\varphi_n \subset B_{\bar{R}\varepsilon_n}(x_n)$ for some $\bar{R}>0$, such that $Q_W(\varphi_n;u_n,G)<0$, with $Q_W$ as in \eqref{b:14}.
		\item[(iv)] For every $R>0$, $q\geq 1$,
		\begin{align*}
		\lim_{n\to\infty}\lambda_n^{\frac12-\frac{q}{p-2}}
		\int_{B_{R\varepsilon_n}(x_n)} u_n^q\,dx
		= \lim_{n\to\infty} \int_{B_R(0)} v_n^q\,dy
		= \int_{B_R(0)} U^q\,dy.
		\end{align*}
	\end{enumerate}
\end{prop}

\begin{proof}
	After rescaling, the vertices of $G$ become invisible in the limit, and a limit problem on the line results. As $G$ has finitely many edges, up to a subsequence all $x_n$ lie on the same edge $e$, giving $(i)$. Set
	\begin{align*}
	\tilde{u}_n(y) = \tilde{\varepsilon}_n^{\frac{2}{p-2}}
	u_n\!\left(x_n+\tilde{\varepsilon}_n y\right), \qquad
	y\in \tilde{e}_n = \frac{e-x_n}{\tilde{\varepsilon}_n}.
	\end{align*}
	Any fixed $[-a,a]$ lies in $\tilde e_n$ for large $n$: indeed
	\begin{align*}
	\tilde e_n \supset \left\{ y:\, |\tilde\varepsilon_n y|
	< \operatorname{dist}(x_n,V) \right\}
	= \left\{ y:\, |y| < \frac{\operatorname{dist}(x_n,V)} {\tilde\varepsilon_n}
	\right\},
	\end{align*}
	and by \eqref{b:33} these intervals exhaust $\mathbb R$. On $[-a,a]$ we have $\tilde u_n(0)=1=\max_{[-a,a]}\tilde u_n$, since $u_n(x_n)=\max_{B_{R_n\tilde\varepsilon_n}(x_n)}u_n$. Moreover $\tilde u_n$ solves
	\begin{align*}
	-\tilde u_n'' +\tilde\varepsilon_n^2
	W(x_n+\tilde\varepsilon_n y)\tilde u_n
	+\tilde\varepsilon_n^2\lambda_n\tilde u_n
	= \tau_n\chi(x_n+\tilde\varepsilon_n y) \tilde u_n^{p-1}
	\end{align*}
	on $\tilde e_n$, with $\tilde u_n>0$.
	By Lemma \ref{bl:4}, $\bigl(u_n(x_n)\bigr)^{p-2}\ge W(x_n)+\lambda_n$, and since
	$W\ge0$ by $(W2)$,
	\begin{align*}
	0<\tilde{\varepsilon}_n^{\,2}\lambda_n
	=\frac{\lambda_n}{\bigl(u_n(x_n)\bigr)^{p-2}}
	\le \frac{\lambda_n+W(x_n)}{\bigl(u_n(x_n)\bigr)^{p-2}}\le 1
	\qquad\forall\,n\in\mathbb N .
	\end{align*}
	Hence, up to a subsequence, $\lim_n\tilde{\varepsilon}_n^{\,2}\lambda_n
	=\bar{\lambda}\in[0,1]$. Since $W\in L^\infty(G)$ and
	$\tilde\varepsilon_n\to0$, we also have
	$\tilde{\varepsilon}_n^{\,2}W(x_n+\tilde{\varepsilon}_ny)\to0$ uniformly on compact sets, and therefore
	\begin{align*}
	\tilde{\varepsilon}_n^{\,2}\bigl(\lambda_n+W(x_n+\tilde{
		\varepsilon}_ny)\bigr) \to\bar{\lambda} \in [0,1].
	\end{align*}
	Elliptic estimates give $\tilde u_n \to \tilde u$ in $C^2_{\mathrm{loc}}(\mathbb R)$, and in the limit
	\begin{equation}\label{b:36}
		-\tilde u''+\bar{\lambda}\,\tilde u = \tilde u^{\,p-1}, \qquad \tilde u\ge 0
		\text{ in } \mathbb R,
	\end{equation}
	with $\tilde u(0)=\lim_n\tilde u_n(0)=1$; the strong maximum principle gives $\tilde u>0$ in $\mathbb R$.
	
	We show $m(\tilde u)\le k^*$. If not, there are $k>k^*$ linearly independent $\varphi_1,\ldots,\varphi_k\in H^1(\mathbb R)\cap C_c(\mathbb R)$ with $Q(\varphi_i;\tilde u,\mathbb R)<0$. Setting $\varphi_{i,n}(x) = \tilde\varepsilon_n^{1/2}\varphi_i\!\big(\tfrac{x-x_n}{\tilde\varepsilon_n}\big)$, each is supported near $x_n$ inside $e$ for large $n$ (with $\operatorname{supp}\varphi_i\subset[-M,M]$,
	$\operatorname{supp}\varphi_{i,n}=[x_n-M\tilde\varepsilon_n,x_n+M\tilde\varepsilon_n]\subset
	[x_n-\operatorname{dist}(x_n,V),x_n+\operatorname{dist}(x_n,V)]\subset e$ by \eqref{b:33}), remains linearly independent in $H^1(G)$, and
	\begin{align*}
	Q_W(\varphi_{i,n};u_n,G) = Q_W(\varphi_{i,n};u_n,e)
	= Q_{W_n}(\varphi_i;\tilde u_n,\tilde e_n)
	\;\xrightarrow[n\to\infty]{}\; Q(\varphi_i;\tilde u,\mathbb R) <0,
	\end{align*}
	where $W_n(y):=\tilde\varepsilon_n^{2}W(x_n+\tilde\varepsilon_n y)\to0$
	uniformly on compact sets, so that the potential disappears in the limit and $Q$ is the $W$-free form. Thus $m(u_n)\ge k>k^*$ for large $n$, contradicting $m(u_n)\le k^*$; hence $m(\tilde u)\le k^*$, and $\tilde u$ is a nontrivial
	solution of \eqref{b:36} of finite Morse index.
	
	We claim $\bar\lambda>0$. If $\bar\lambda=0$, then \eqref{b:36} reduces to
	\begin{align*}
	-\tilde u'' = \tilde u^{\,p-1} \qquad\text{in } \mathbb R,
	\end{align*}
	which, by phase-plane analysis, has only nontrivial periodic sign-changing solutions --- a contradiction. Hence $\bar{\lambda}>0$, and by Lemma \ref{bl:5} $\tilde u(x)\to0$ and $\tilde u\in H^1(\mathbb R)$. Therefore,
	\begin{equation}\label{b:37}
		0< \liminf_{n\to\infty} \frac{\lambda_n}
		{\bigl(u_n(x_n)\bigr)^{p-2}} \le \limsup_{n\to\infty}
		\frac{\lambda_n} {\bigl(u_n(x_n)\bigr)^{p-2}} \le 1,
	\end{equation}
	i.e., $0<\liminf \tilde{\varepsilon}_n^{2}\lambda_n\le \limsup \tilde{\varepsilon}_n^{2}\lambda_n\le 1$, which is the first estimate in \eqref{b:34}. Working now with $v_n$ from \eqref{b:35} and combining \eqref{b:33} and \eqref{b:37}, $\limsup_n \operatorname{dist}(x_n,V)/\varepsilon_n = +\infty$, and $v_n\to v$ in $C^2_{\mathrm{loc}}(\mathbb R)$ with $-v''+v=v^{p-1}$, $v\ge0$; moreover $v(0)=\lim_n v_n(0)\ge1$ is a global maximum, so $v>0$ by the strong maximum principle, and $m(v)\le k^*$. By Lemma \ref{bl:5}, $v(x)\to0$ and $v\in H^1(\mathbb R)$; such a solution is unique, denoted $U$, proving $(ii)$.
	
	Assertion $(iii)$ follows since $U$ has strictly positive Morse index (Lemma \ref{bl:6}); in fact $m(U)=1$, so there exists $\varphi\in H^1(\mathbb R)\cap C_c(\mathbb R)$ with $Q(\varphi;U,\mathbb R)<0$, where $Q(\cdot;U,\mathbb R)$ is the quadratic form associated with the autonomous
	limit equation $-U''+U=U^{p-1}$, namely
	\begin{align*}
	Q(\varphi;U,\mathbb R)
	= \int_{\mathbb R}\Big(|\varphi'|^2 + \big(1-(p-1)U^{p-2}\big)\varphi^2\Big)\,dx .
	\end{align*}
	Define the rescaled test function
	\begin{align*}
	\varphi_{n}(x) = \varepsilon_n^{1/2}\,\varphi\!\Big(\tfrac{x-x_n}{\varepsilon_n}\Big),
	\end{align*}
	which, since $\varphi$ has compact support, satisfies
	$\operatorname{supp}\varphi_{n} \subset B_{\bar R\varepsilon_n}(x_n)$ for some $\bar R>0$ and, for $n$ large, is supported inside the edge $e$; hence it may be
	regarded as an element of $H^1(G)\cap C_c(G)$. By the change of variables $x=x_n+\varepsilon_n y$, using $\varepsilon_n^2\lambda_n=1$, $\varepsilon_n^{2}W(x_n+\varepsilon_n y)\to0$ locally uniformly, and $v_n\to U$ in $C^2_{\mathrm{loc}}(\mathbb R)$, we obtain
	\begin{align*}
	Q_W(\varphi_{n};u_n,G) \;\xrightarrow[n\to\infty]{}\; Q(\varphi;U,\mathbb R) < 0 .
	\end{align*}
	Therefore $Q_W(\varphi_{n};u_n,G)<0$ for all sufficiently large $n$, which establishes $(iii)$.
	
	Finally $(iv)$ follows from $v_n\to U$ in $C^2_{\mathrm{loc}}(\mathbb R)$: for $R>0$, $q\ge1$, $v_n^q\to U^q$ uniformly on $B_R(0)$, so $\int_{B_R(0)}v_n^q\,dy\to\int_{B_R(0)}U^q\,dy$. With $x=x_n+\varepsilon_n y$,
	\begin{align*}
	\int_{B_R(0)} v_n^q\,dy = \lambda_n^{\frac12 -\frac{q}{p-2}}
	\int_{B_{R\varepsilon_n}(x_n)} u_n^q\,dx,
	\end{align*}
	whence
	\begin{align*}
	\lim_{n\to\infty} \lambda_n^{\frac12-\frac{q}{p-2}}
	\int_{B_{R\varepsilon_n}(x_n)} u_n^q\,dx = \int_{B_R(0)} U^q\,dy.
	\end{align*}
	This proves $(iv)$.  \qed
\end{proof}

\begin{prop}\label{bp:3}
	Assume all hypotheses of Proposition \ref{bp:2} except \eqref{b:33}, and instead
	\begin{equation}\label{b:38}
		\limsup_{n\to\infty} \frac{\operatorname{dist}(x_n,V)}
		{\tilde{\varepsilon}_n} <+\infty.
	\end{equation}
	Then, up to a subsequence:
	\begin{enumerate}
		\item[(i)] There is a vertex $v\in V$ with $x_n \to v$; all $x_n$ lie on the same edge $e_1 \simeq [0,\ell_1]$, $e_1\subset \mathcal{K}$, with $v$ at coordinate $0$.
		\item[(ii)] Let $e_2 \simeq [0,\ell_2],\ldots,e_j \simeq [0,\ell_j]$ be the other bounded edges at $v$ and $\ell_{j+1}\simeq [0,+\infty),\ldots,\ell_m\simeq
		[0,+\infty)$ the half-lines at $v$, with $v$ at $0$ on each. With $\varepsilon_n=\lambda_n^{-1/2}$,
		\begin{equation}\label{b:39}
			\frac{\tilde{\varepsilon}_n}{\varepsilon_n}
			\to c\in(0,1],\qquad \limsup_{n\to\infty}
			\frac{\operatorname{dist}(x_n,V)} {\varepsilon_n} <+\infty,
		\end{equation}
		and
		\begin{align*}
		v_n(y) = \varepsilon_n^{\frac{2}{p-2}} u_n(\varepsilon_n y), \quad y \in \tfrac{e_i}{\varepsilon_n}\ (i=1,\ldots,j)\ \text{ or }\ y \in\tfrac{\ell_i}{\varepsilon_n}\ (i=j+1,\ldots,m),
		\end{align*}
		converges to $U\in C^0_{\mathrm{loc}}(G_m)$; with $U_i=U|_{\ell_i}$ and $v_{i,n}$ the restriction of $v_n$ to $e_i/\varepsilon_n$ (resp.\ to $\ell_i/\varepsilon_n$ for $i>j$), one has $v_{i,n}\to U_i$ in $C^2_{\mathrm{loc}}([0,+\infty))$ for all $i$. Moreover
		$U\in H^1(G_m)$ is a positive solution of
		\begin{align*}
		\begin{cases}
			-U''+U = U^{p-1}\chi_{\ell_1\cup\cdots\cup\ell_j} & \text{on } G_m,\\[2mm]
			U>0 & \text{on } G_m,\\[2mm]
			\displaystyle\sum_{i=1}^{m}U_i'(0^+)=0,\\[2mm]
			U(x)\to 0 & \text{as } \operatorname{dist}(x,0)\to+\infty,
		\end{cases}
		\end{align*}
		attaining its global maximum at $\bar x\in \ell_1$, with
		$\bar x = \lim_n\bar x_n\in[0,+\infty)$, $\bar x_n = \operatorname{dist}(x_n,V)/\varepsilon_n$.
		\item[(iii)] There is $\varphi_n \in H^1(G)\cap C_c(G)$ with
		$\operatorname{supp}\varphi_n \subset B_{\bar R\,\varepsilon_n}(x_n)$, for some $\bar R>0$,
		such that $Q_W(\varphi_n;u_n,G)<0$.
		\item[(iv)] For $R>\bar x$ and $q\ge 1$,
		\begin{align*}
			\lim_{n\to\infty} \lambda_n^{\frac12-\frac{q}{p-2}}
			\int_{B_{R\varepsilon_n}(x_n)} u_n^q\,dx
			&= \lim_{n\to\infty} \int_{B_R(\bar x_n)} v_n^q\,dy \\
			&= \int_{[0,\bar x+R]} U_1^q\,dy + \sum_{i=2}^{m} \int_{[0,R-\bar x]} U_i^q\,dy
			= \int_{B_R(\bar x)} U^q\,dy.
		\end{align*}
	\end{enumerate}
\end{prop}

\begin{proof}
	Under \eqref{b:38}, $\tilde{\varepsilon}_n \to 0$ forces $x_n \to v$ for some $v \in V$. Unlike \cite{chang2023normalized}, $v$ may lie on a half-line of $G$, so \cite[Theorem 4.2]{chang2023normalized} cannot be applied directly but adapts
	with minor changes; this gives $(i)$. Assume $\tfrac{t_n}{\tilde{\varepsilon}_n} \to \beta \in [0,+\infty)$, with
	$t_n = \operatorname{dist}(x_n,V) = x_n$. Define
	\begin{align*}
	\tilde{u}_{i,n}(y) = \tilde{\varepsilon}_n^{\frac{2}{p-2}}
	u_n(\tilde{\varepsilon}_n y),\qquad
	y\in \tfrac{e_i}{\tilde{\varepsilon}_n}\ (i=1,\ldots,j)\ \text{ or }\ y\in \tfrac{\ell_i}{\tilde{\varepsilon}_n}\ (i=j+1,\ldots,m),
	\end{align*}
	and $\tilde{u}_n = (\tilde{u}_{1,n},\ldots,\tilde{u}_{m,n})$, defined on a graph $G_{m,n}$ of $j$ expanding edges and $m-j$ half-lines joined at $0$, with $G_{m,n}\to G_m$. For $a>\beta+1$ and large $n$,
	\begin{align*}
	\tilde{u}_{1,n}\!\left(\tfrac{t_n}{\tilde{\varepsilon}_n}\right)
	= \tilde{\varepsilon}_n^{\frac{2}{p-2}} u_n(x_n) =1,
	\end{align*}
	so
	\begin{equation}\label{b:40}
		\tilde{u}_{1,n}\!\left(\tfrac{t_n}{\tilde{\varepsilon}_n}\right) =1
		=\max_{B_a(0)} \tilde{u}_n,
	\end{equation}
	and $\tilde u_n$ solves
	\begin{align*}
	-\tilde u_n'' +\tilde\varepsilon_n^2 W(\tilde\varepsilon_n y)\tilde u_n +\tilde\varepsilon_n^2\lambda_n\tilde u_n
	= \tau_n \tilde{u}_n^{\,p-1} \chi_{\tilde e_{1,n}\cup\cdots\cup\tilde e_{j,n}},
	\quad \tilde{u}_n >0,
	\end{align*}
	with Kirchhoff condition at $0$. By Lemma \ref{bl:4} we have
	$\bigl(u_n(x_n)\bigr)^{p-2}\ge W(x_n)+\lambda_n$, and since $W\ge 0$ by $(W2)$,
	\begin{align*}
	0<\tilde{\varepsilon}_n^{\,2}\lambda_n
	=\frac{\lambda_n}{\bigl(u_n(x_n)\bigr)^{p-2}}
	\;\le\; \frac{\lambda_n+W(x_n)}{\bigl(u_n(x_n)\bigr)^{p-2}}\;\le\; 1,
	\qquad \forall\, n\in\mathbb N .
	\end{align*}
	Hence, passing to a subsequence if necessary, we may assume
	$\lim_n\tilde{\varepsilon}_n^{\,2}\lambda_n = \bar{\lambda}\in [0,1]$. Since
	$W\in L^\infty(G)$ and $\tilde{\varepsilon}_n\to0$, we have
	$\tilde{\varepsilon}_n^{\,2}W(\tilde{\varepsilon}_ny)\to0$ uniformly on compact
	sets, and therefore
	\begin{align*}
	\tilde{\varepsilon}_n^{\,2}\bigl(\lambda_n+W(\tilde{\varepsilon}_ny)\bigr) \to\bar{\lambda} \in [0,1].
	\end{align*}
	Elliptic estimates give, for $i=1,\ldots,j$, $\tilde u_{i,n}\to\tilde u_i$ in $C^2_{\mathrm{loc}}([0,+\infty))$, with
	\begin{equation}\label{b:41}
		-\tilde u_i''+\bar{\lambda}\,\tilde u_i = \tilde u_i^{\,p-1}, \qquad \tilde u_i\ge 0
		\quad \text{in } (0,+\infty),
	\end{equation}
	and for $i=j+1,\ldots,m$, $\tilde{u}_{i,n}\to\tilde{u}_i$ with $-\tilde{u}_i''+\bar{\lambda}\,\tilde{u}_i=0$, $\tilde u_i\ge0$. Continuity of $\tilde u_n$ and uniform convergence give $\tilde{u}_{i_1}(0)=\tilde{u}_{i_2}(0)$ for all $i_1\ne i_2$, so $\tilde u=(\tilde u_1,\ldots,\tilde u_m)$ is defined on 	$G_m$, and the Kirchhoff condition passes to the limit. Since 	$\tilde u_1(\beta)=\lim_n\tilde u_{1,n}(t_n/\tilde\varepsilon_n)=1$ is a global maximum by \eqref{b:40}, ODE uniqueness gives $\tilde u_1>0$ on $(0,+\infty)$,and then, via the Kirchhoff condition, $\tilde u_i>0$ on $(0,+\infty)$ for all $i$.
	
	We claim $m(\tilde u)\le k^*$, arguing as in Proposition \ref{bp:2} and using $m(u_n)\le k^*$. If $\bar{\lambda}=0$, \eqref{b:41} becomes 	
    \begin{align*} 	
    -\tilde u_1'' = \tilde u_1^{\,p-1}, \qquad \text{on } (0,+\infty), \end{align*} 
    which, by phase-plane analysis, has only nontrivial periodic sign-changing solutions --- a contradiction. Hence $\bar{\lambda}>0$, so
	\begin{equation}\label{b:42}
		0< \liminf_{n\to\infty} \frac{\lambda_n}
		{\bigl(u_n(x_n)\bigr)^{p-2}} \le \limsup_{n\to\infty}
		\frac{\lambda_n} {\bigl(u_n(x_n)\bigr)^{p-2}} \le 1,
	\end{equation}
	which, since $\tilde\varepsilon_n/\varepsilon_n=\bigl(\tilde\varepsilon_n^{\,2}\lambda_n\bigr)^{1/2}\to\bar\lambda^{1/2}=:c$,
	is the first estimate in \eqref{b:39}. By Lemma \ref{bl:5}, $\tilde u(x)\to0$ and $\tilde u\in H^1(G_m)$. Passing to $v_n$ and combining \eqref{b:38} with \eqref{b:42}, $\limsup_n \frac{\operatorname{dist}(x_n,V)}{\varepsilon_n}<+\infty$, and $v_n\to U$ in $C^0_{\mathrm{loc}}(G_m)$ and in
	$C^2_{\mathrm{loc}}([0,+\infty))$ on each half-line, with limit $U=(U_1,\ldots,U_m)$ solving
	\begin{align*}
	\begin{cases}
		- U'' + U = U^{p-1}\chi_{\ell_1\cup\cdots\cup\ell_j} & \text{on } G_m,\\[2mm]
		U\ge 0 & \text{on } G_m,\\[2mm]
		\displaystyle \sum_{i=1}^{m} U_i'(0^+) =0.
	\end{cases}
	\end{align*}
	Moreover $U$ has a strictly positive global maximum at $\bar x\in\ell_1$ with $U_1(\bar x)\ge1$, so $U$ is nontrivial and, by the strong maximum principle and ODE uniqueness, $U>0$ on $G_m$; also $m(U)\le k^*$. By Lemma \ref{bl:5}, $U(x)\to0$, proving $(ii)$.
	
	Assertion $(iii)$ follows from Lemma \ref{bl:6}: there exists
	$\varphi\in H^1(G_m)\cap C_c(G_m)$ with $Q(\varphi;U,G_m)<0$,
	where $Q(\cdot\,;U,G_m)$ denotes the quadratic form associated
	with the autonomous limit problem on $G_m$, namely
	\begin{align*}
	Q(\varphi;U,G_m)=\int_{G_m}\Bigl(|\varphi'|^2
	+\bigl(1-(p-1)U^{p-2}\chi_{\ell_1\cup\cdots\cup\ell_j}\bigr)\varphi^2\Bigr)\,dy .
	\end{align*}
	Define the rescaled test function
	\begin{align*}
	\varphi_{n}(x) = \varepsilon_n^{\frac12}\,
	\varphi\!\Bigl(\tfrac{x}{\varepsilon_n}\Bigr),
	\end{align*}
	which, $\varphi$ having compact support in $G_m$, is supported in $B_{\bar R\varepsilon_n}(x_n)$ for some $\bar R>0$ and, for $n$ large, lies in the union of the edges incident at $v$; hence $\varphi_n\in H^1(G)\cap C_c(G)$. By
	the change of variables $x=\varepsilon_n y$, using $\varepsilon_n^{2}\lambda_n=1$, $\varepsilon_n^{2}W(\varepsilon_n y)\to0$ uniformly on compact sets, and $v_n\to U$ in $C^0_{\mathrm{loc}}(G_m)$, we obtain
	\begin{align*}
	Q_W(\varphi_{n};u_n,G)\;\xrightarrow[n\to\infty]{}\;Q(\varphi;U,G_m)<0 .
	\end{align*}
	Therefore $Q_W(\varphi_{n};u_n,G)<0$ for all sufficiently large $n$, which establishes $(iii)$.
	
	Finally, $(iv)$ follows from the local uniform convergence $v_n\to U$: for $R>\bar x$ and $q\ge1$, the change of variables $x=\varepsilon_n y$ together with
	$\varepsilon_n=\lambda_n^{-1/2}$ gives
	\begin{align*}
	\int_{B_R(\bar x_n)} v_n^q\,dy
	=\lambda_n^{\frac12-\frac{q}{p-2}}\int_{B_{R\varepsilon_n}(x_n)} u_n^q\,dx ,
	\end{align*}
	whence
	\begin{align*}
	\lim_{n\to\infty} \lambda_n^{\frac12-\frac{q}{p-2}}
	\int_{B_{R\varepsilon_n}(x_n)} u_n^q\,dx
	= \int_{B_R(\bar{x})} U^q\,dy
	= \int_{[0,\bar{x}+R]} U_1^q\,dy + \sum_{i=2}^{m} \int_{[0,R-\bar{x}]} U_i^q\,dy .
	\end{align*}
	This proves $(iv)$.  \qed
\end{proof}

\begin{prop}\label{bp:4}
	Let $G$ be a noncompact metric graph with finitely many edges and vertices, let $W$ satisfy $(W1)$--{$(W2)$}, and let $p>2$. Let $\{u_n\}\subset H^1(G)$ be a sequence of positive solutions of \eqref{b:31} with $\lambda_n\to+\infty$ and $m(u_n)\le 2$ for every $n\in\mathbb{N}$. Then there exist $k\in\{1,2\}$, sequences of points $\{x_n^i\}$, $i=1,\ldots,k$, and constants $R,C_1,C_2>0$ independent of $n$, such that
    	\begin{align*}
	\lambda_n^{1/2}\,\operatorname{dist}(x_n^1,x_n^2)\to+\infty,
	\qquad\text{if } k=2,
	\end{align*}
	\begin{align*}
	u_n(x_n^i) = \max_{B_{R_n\lambda_n^{-1/2}}(x_n^i)} u_n,
	\qquad \text{for all } i=1,\ldots,k,
	\end{align*}
	for some sequence $R_n\to+\infty$, and
	\begin{equation}\label{b:43}
		u_n(x) \le C_1\lambda_n^{\frac1{p-2}}
		\sum_{i=1}^{k} e^{-C_2\lambda_n^{1/2}\operatorname{dist}(x,x_n^i)}
		+ C_1\lambda_n^{\frac1{p-2}} \sum_{j=1}^{m_3}
		e^{-C_2\lambda_n^{1/2}\operatorname{dist}(x,v_j)},
	\end{equation}
	for every $x\in G\setminus \bigcup_{i=1}^{k} B_{R\lambda_n^{-1/2}}(x_n^i)$,
	where $v_1,\ldots,v_{m_3}$ denote all the vertices of $G$.
\end{prop}
\begin{proof}
	Throughout, we set
	\begin{align*}
	d_n(x) := \min\bigl\{\operatorname{dist}(x,x_n^i) \,:\, i=1,\ldots,k\bigr\},
	\qquad
	A_{n,R} := \bigl\{x\in G \,:\, d_n(x) > R\lambda_n^{-1/2}\bigr\}.
	\end{align*}
	
	\medskip
	\noindent\textbf{Step 1: Selection of the blow-up points and uniform smallness.} We claim that there exist $k\in\{1,2\}$ and sequences $\{x_n^i\}$, $i=1,\ldots,k$, satisfying the two displayed conditions of the statement, together with
	\begin{equation}\label{b:44}
		\lim_{R\to+\infty}
		\left( \limsup_{n\to\infty}\;
		\lambda_n^{-\frac{1}{p-2}} \max_{d_n(x)\ge R\lambda_n^{-1/2}} u_n(x) \right) = 0.
	\end{equation}
	By Lemma \ref{bl:4} every local maximum point of $u_n$ lies on $\mathcal{K}$, and Propositions \ref{bp:2} and \ref{bp:3} describe the blow-up profile of $u_n$ near any such point: after rescaling by $\varepsilon_n=\lambda_n^{-1/2}$, the potential term $\varepsilon_n^2\, W(x_n+\varepsilon_n\,\cdot\,)$ vanishes in the limit, and the profile solves the autonomous equation on $\mathbb{R}$ (Proposition \ref{bp:2}) or on a star graph $G_m$ (Proposition \ref{bp:3}). Each such profile carries a
	negative direction of the quadratic form supported in a ball
	$B_{\bar R\varepsilon_n}(x_n)$, by Propositions \ref{bp:2}$(iii)$ and \ref{bp:3}$(iii)$.
    	These balls are mutually disjoint whenever $\lambda_n^{1/2}\operatorname{dist}(x_n^i,x_n^{i'})\to+\infty$, so each concentration point contributes at least one dimension to the Morse index of $u_n$; as $m(u_n)\le 2$, at most two such points can occur, whence $k\in\{1,2\}$. Given this, the selection procedure and the uniform decay \eqref{b:44} follow by adapting \cite[Theorem 3.2]{esposito2011pointwise}, exactly as in \cite[Theorem 4.6, Step 1]{chang2023normalized}; the presence of $W$ plays no role here, since $W\in L^\infty(G)$ and $\varepsilon_n^2\|W\|_{L^\infty(G)}\to0$.
	
	In particular, \eqref{b:44} yields: for every $\varepsilon\in(0,1)$ there exist
	$R>0$ and $n_R\in\mathbb{N}$ such that
	\begin{equation}\label{b:45}
		\max_{d_n(x)> R\lambda_n^{-1/2}} u_n(x) \;\le\; \varepsilon\,\lambda_n^{\frac{1}{p-2}},
		\qquad \forall\, n\ge n_R .
	\end{equation}
	Moreover, by \eqref{b:37} (resp.\ \eqref{b:42}) we have
	$\bigl(u_n(x_n^i)\bigr)^{p-2}\le C\,\lambda_n$ for some $C>0$ independent of $n$,
	so that
	\begin{equation}\label{b:46}
		u_n(x) \;\le\; C_0\,\lambda_n^{\frac{1}{p-2}},
		\qquad \forall\, x\in G,\ \forall\, n\in\mathbb{N},
	\end{equation}
	for a constant $C_0>0$ independent of $n$.
	
	\medskip
	\noindent\textbf{Step 2: The differential inequality on $A_{n,R}$.} We claim that, choosing $\varepsilon\in(0,1)$ in \eqref{b:45} so small that $\varepsilon^{p-2}\le\tfrac12$, and $R,n$ correspondingly large,
	\begin{equation}\label{b:47}
		-u_n'' + \frac{\lambda_n}{2}\,u_n \;\le\; 0
		\qquad \text{on } A_{n,R}\cap e, \text{ for every edge } e\in E .
	\end{equation}
	
	\emph{Case 1: $x\in \mathcal{K}\cap A_{n,R}$.} Here $\chi(x)=1$, so \eqref{b:31} gives
	\begin{align*}
	u_n''(x) = \bigl(W(x)+\lambda_n-\tau_n u_n(x)^{p-2}\bigr)\,u_n(x).
	\end{align*}
	By \eqref{b:45} and $\varepsilon^{p-2}\le\tfrac12$ we have
	$u_n(x)^{p-2}\le \varepsilon^{p-2}\lambda_n\le \tfrac{\lambda_n}{2}$, and since
	$\tau_n\le 1$,
	\begin{align*}
	W(x)+\lambda_n-\tau_n u_n(x)^{p-2}
	\;\ge\;  \lambda_n - \frac{\lambda_n}{2}
	\; =\; \frac{\lambda_n}{2}.
	\end{align*}
	As $u_n>0$, this yields $u_n''(x)\ge \tfrac{\lambda_n}{2}u_n(x)$, i.e.\
	\eqref{b:47}.
	
	\emph{Case 2: $x\in G\setminus\mathcal{K}$.} Here $\chi(x)=0$, so \eqref{b:31}
	gives
	\begin{align*}
	u_n''(x) = \bigl(W(x)+\lambda_n\bigr)u_n(x) \;\ge\; \lambda_n u_n(x)
	\;\ge\; \frac{\lambda_n}{2}\,u_n(x),
	\end{align*}
	using $W\ge0$ by $(W2)$, $u_n>0$ and $\lambda_n>0$. Again \eqref{b:47} holds.
	
	In both cases the nonnegativity of $W$ acts in our favour, so that \eqref{b:47}
	is obtained under weaker requirements than in the autonomous case.
	
	\medskip
	\noindent\textbf{Step 3: Comparison with exponential barriers.} Fix $C_2:=\tfrac12$ and let $e\in E$ be an edge, identified with $[0,\ell_e]$ or with $[0,+\infty)$. The set $A_{n, R}\cap e$ is a finite union of at most $k+1$ relatively open subintervals of $e$. Let $I=(a,b)$ be one of them, with $b\le+\infty$. By construction, each finite endpoint of $I$ is either a vertex of
	$G$, or a point of $\partial B_{R\lambda_n^{-1/2}}(x_n^i)$ for some $i\in\{1,\ldots,k\}$; in either case it belongs to $\overline{A_{n,R}}$, so that \eqref{b:45}--\eqref{b:46} give
	\begin{equation}\label{b:48}
		u_n \;\le\; \Lambda_n := C_0\,\lambda_n^{\frac{1}{p-2}}
		\qquad\text{at every finite endpoint of } I .
	\end{equation}
	
	Assume first $b<+\infty$, and define on $\overline I$ the barrier
	\begin{align*}
	\Phi_n(t) := \Lambda_n\Bigl[
	e^{-C_2\lambda_n^{1/2}(t-a)} + e^{-C_2\lambda_n^{1/2}(b-t)} \Bigr].
	\end{align*}
	A direct computation gives $\Phi_n'' = C_2^2\lambda_n\Phi_n
	= \tfrac{\lambda_n}{4}\Phi_n$, hence
	\begin{align*}
	-\Phi_n'' + \frac{\lambda_n}{2}\Phi_n
	= \Bigl(\frac{\lambda_n}{2}-\frac{\lambda_n}{4}\Bigr)\Phi_n
	= \frac{\lambda_n}{4}\,\Phi_n \;>\;0 \qquad\text{on } I,
	\end{align*}
	while $\Phi_n(a)\ge\Lambda_n$ and $\Phi_n(b)\ge\Lambda_n$, so that by \eqref{b:48} $\Phi_n\ge u_n$ at both endpoints. Setting $w:=u_n-\Phi_n$, we obtain from \eqref{b:47}
	\begin{align*}
	-w'' + \frac{\lambda_n}{2}\,w \;\le\; 0 \quad\text{on } I,
	\qquad w\le 0 \ \text{ on } \partial I .
	\end{align*}
	Suppose $w$ attained a positive maximum at some $t_0\in I$. Then $w(t_0)>0$ and $w''(t_0)\le0$, whence $-w''(t_0)+\tfrac{\lambda_n}{2}w(t_0)>0$, a contradiction. Therefore $w\le0$, i.e.
	\begin{equation}\label{b:49}
		u_n(t) \;\le\; \Lambda_n\Bigl[
		e^{-C_2\lambda_n^{1/2}(t-a)} + e^{-C_2\lambda_n^{1/2}(b-t)} \Bigr],
		\qquad \forall\, t\in \overline I .
	\end{equation}
	
	If instead $b=+\infty$ (so that $e$ is a half-line), we use the barrier $\Phi_n(t):=\Lambda_n e^{-C_2\lambda_n^{1/2}(t-a)}$, which again satisfies $-\Phi_n''+\tfrac{\lambda_n}{2}\Phi_n>0$ on $I$ and $\Phi_n(a)\ge u_n(a)$. Since
	$u_n\in H^1(e)$ we have $u_n(t)\to0$ as $t\to+\infty$, and also $\Phi_n(t)\to 0$; hence $w=u_n-\Phi_n$ satisfies $\limsup_{t\to+\infty}w(t)\le0$, and the same argument gives $w\le0$ on $\overline I$, i.e.\ \eqref{b:49} without the second exponential.
	
	\medskip
	\noindent\textbf{Step 4: Conclusion.}
	It remains to rewrite \eqref{b:49} intrinsically. Let
	$x\in G\setminus\bigcup_{i=1}^{k}B_{R\lambda_n^{-1/2}}(x_n^i)$, and let $I=(a,b)$ be the component of $A_{n,R}\cap e$ containing $x$, with $t$ the coordinate of $x$ on $e$.
	
	If the endpoint $a$ is a vertex $v_j$ of $G$, then $t-a=\operatorname{dist}(x,v_j)$, so that
	\begin{align*}
	\Lambda_n e^{-C_2\lambda_n^{1/2}(t-a)}
	= C_0\,\lambda_n^{\frac{1}{p-2}}\,e^{-C_2\lambda_n^{1/2}\operatorname{dist}(x,v_j)} .
	\end{align*}
	If instead $a\in\partial B_{R\lambda_n^{-1/2}}(x_n^i)$, then
	$t-a=\operatorname{dist}(x,x_n^i)-R\lambda_n^{-1/2}$, whence
	\begin{align*}
	\Lambda_n e^{-C_2\lambda_n^{1/2}(t-a)}
	= C_0\,e^{C_2R}\,\lambda_n^{\frac{1}{p-2}}\,
	e^{-C_2\lambda_n^{1/2}\operatorname{dist}(x,x_n^i)} .
	\end{align*}
	The same holds for the endpoint $b$. Since $R$ is fixed and $G$ has finitely many edges and vertices, setting $C_1:=2\,C_0\,e^{C_2R}$ and summing the resulting
	contributions over the (at most two) blow-up points and the $m_3$ vertices, we obtain \eqref{b:43}, namely
	\begin{align*}
	u_n(x) \le C_1\lambda_n^{\frac1{p-2}}
	\sum_{i=1}^{k} e^{-C_2\lambda_n^{1/2}\operatorname{dist}(x,x_n^i)}
	+ C_1\lambda_n^{\frac1{p-2}} \sum_{j=1}^{m_3}
	e^{-C_2\lambda_n^{1/2}\operatorname{dist}(x,v_j)},
	\end{align*}
	for every $x\in G\setminus\bigcup_{i=1}^{k}B_{R\lambda_n^{-1/2}}(x_n^i)$, with
	$C_1,C_2>0$ independent of $n$. This completes the proof.  \qed
\end{proof}
\begin{rem}
		Two features of this proof deserve comment. Only $(W1)$ and the nonnegativity $(W2)$ of the potential are used: assumption $(W3)$, essential for the mountain-pass geometry (Lemma \ref{bl:2}) and for the test-space construction (Lemma \ref{bl:3}), and assumption $(W4)$, essential for the sign of $\lambda_\tau$ (Proposition \ref{bp:1}), play no role in the decay estimate. And the sign of $W$ acts in our favour in Step 2: on $\mathcal K\cap A_{n,R}$ it only increases the coefficient $W+\lambda_n-\tau_nu_n^{p-2}$, while on $G\setminus\mathcal K$ it gives $u_n''=(W+\lambda_n)u_n\ge\lambda_nu_n$ at once, so that \eqref{b:47} holds under weaker smallness requirements than in the autonomous case treated in \cite{borthwick2023normalized}. Without a sign condition on $W$, the corresponding estimate requires the maximum principle of \cite[Proposition 2.11]{carrillo2026blow}.
\end{rem}

We can now prove Theorem \ref{bt:2}.

\medskip
\noindent\textbf{Proof of Theorem \ref{bt:2}.}
By Proposition \ref{bp:1} there exist a sequence $\tau_n\to1^-$ and corresponding constrained critical points $u_n:=u_{\tau_n}\in H^1_\mu(G)$ of $E_{W,\tau_n}(\cdot,G)$ on $H^1_\mu(G)$, lying at the mountain-pass levels $c_{W,\tau_n}$, with $u_n>0$ on $G$, $m(u_n)\le2$, and associated Lagrange multipliers $\lambda_n>0$ for every $n\in\mathbb{N}$.

\medskip
\noindent\textbf{Step 1: Boundedness of $\{\lambda_n\}$.}
We have already observed that if $\{\lambda_n\}$ is bounded, then $\{u_n\}$ is bounded in $H^1(G)$. Assume by contradiction that, up to a subsequence, $\lambda_n \to +\infty$. Then Propositions \ref{bp:2}, \ref{bp:3} and \ref{bp:4} apply to $\{u_n\}$. Let $k\in\{1,2\}$ and $\{x_n^i\}_{i=1}^k$ be the blow-up points given by Proposition \ref{bp:4}, and fix $R>0$ as in its statement. We estimate the mass of $u_n$ by splitting
\begin{equation}\label{b:50}
	\mu = \int_G |u_n|^2\,dx
	= \sum_{i=1}^{k}\int_{B_{R\varepsilon_n}(x_n^i)} |u_n|^2\,dx
	\;+\; \int_{G\setminus\bigcup_{i=1}^{k}B_{R\varepsilon_n}(x_n^i)} |u_n|^2\,dx ,
\end{equation}
where $\varepsilon_n=\lambda_n^{-1/2}$.

\emph{Contribution of the blow-up balls.} Applying
Proposition \ref{bp:2}$(iv)$ (resp.\ Proposition \ref{bp:3}$(iv)$) with $q=2$, we obtain, for each $i=1,\ldots,k$,
\begin{align*}
\lambda_n^{\frac12-\frac{2}{p-2}}
\int_{B_{R\varepsilon_n}(x_n^i)} |u_n|^2\,dx
\;\longrightarrow\; \int_{B_R(\bar x)} U^2\,dy \;<\;+\infty ,
\end{align*}
whence
\begin{equation}\label{b:51}
	\int_{B_{R\varepsilon_n}(x_n^i)} |u_n|^2\,dx
	\;\le\; C\,\lambda_n^{\frac{2}{p-2}-\frac12}
\end{equation}
for some $C>0$ independent of $n$. Since $p>6$, we have $p-2>4$ and therefore
\begin{equation}\label{b:52}
	\frac{2}{p-2}-\frac12 \;<\; 0 .
\end{equation}
As $\lambda_n\to+\infty$, the right-hand side of \eqref{b:51} tends to $0$, so that
\begin{align*}
\sum_{i=1}^{k}\int_{B_{R\varepsilon_n}(x_n^i)} |u_n|^2\,dx
\;\longrightarrow\; 0 \qquad (n\to\infty).
\end{align*}

\emph{Contribution of the exterior region.} Set	$\Omega_n:=G\setminus\bigcup_{i=1}^{k}B_{R\varepsilon_n}(x_n^i)$. By Proposition \ref{bp:4}, and using $\bigl(\sum_{\ell=1}^{N}a_\ell\bigr)^2\le N\sum_{\ell=1}^{N}a_\ell^2$, we obtain for every $x\in\Omega_n$
\begin{align*}
|u_n(x)|^2 \;\le\; C\,\lambda_n^{\frac{2}{p-2}}
\left[ \sum_{i=1}^{k} e^{-2C_2\lambda_n^{1/2}\operatorname{dist}(x,x_n^i)}
+ \sum_{j=1}^{m_3} e^{-2C_2\lambda_n^{1/2}\operatorname{dist}(x,v_j)} \right],
\end{align*}
with $C=C(C_1,k,m_3)>0$ independent of $n$. For any $y\in G$, since $G$ has finitely many edges, the measure of $\{x\in G:\operatorname{dist}(x,y)\in[t,t+dt]\}$ is bounded by $|E|\,dt$, whence
\begin{equation}\label{b:53}
	\int_G e^{-2C_2\lambda_n^{1/2}\operatorname{dist}(x,y)}\,dx
	\;\le\; |E|\int_0^{+\infty} e^{-2C_2\lambda_n^{1/2}t}\,dt
	\;=\; \frac{|E|}{2C_2}\,\lambda_n^{-\frac12}.
\end{equation}
Applying \eqref{b:53} with $y=x_n^i$, $i=1,\ldots,k$, and with $y=v_j$, $j=1,\ldots,m_3$, we conclude
\begin{align*}
\int_{\Omega_n} |u_n|^2\,dx
\;\le\; C\,\lambda_n^{\frac{2}{p-2}}\,(k+m_3)\,\frac{|E|}{2C_2}\,\lambda_n^{-\frac12}
\;=\; \tilde C\,\lambda_n^{\frac{2}{p-2}-\frac12} \;\longrightarrow\; 0 ,
\end{align*}
again by \eqref{b:52} and $\lambda_n\to+\infty$, with $\tilde C>0$ independent of $n$.

\emph{Conclusion of Step 1.} Passing to the limit in \eqref{b:50}, both terms on the right-hand side vanish, so that
\begin{align*}
\mu = \lim_{n\to\infty}\int_G |u_n|^2\,dx = 0 ,
\end{align*}
contradicting $\mu>0$. Hence the assumption $\lambda_n\to+\infty$ is impossible, and $\{\lambda_n\}$ is bounded. Consequently $\{u_n\}$ is bounded in $H^1(G)$.

\medskip
\noindent\textbf{Step 2: Passage to the limit.}
Since $\{\lambda_n\}$ is bounded, up to a subsequence there exists $\hat\lambda\in\mathbb{R}$ with $\lambda_n\to\hat\lambda$; moreover $\hat\lambda\ge0$, because $\lambda_n>0$ for every $n$. As $\{u_n\}$ is bounded in $H^1(G)$ with $u_n>0$, there exists $\hat u\in H^1(G)$ such that, up to a further subsequence,
\begin{align*}
	u_n &\rightharpoonup \hat u && \text{weakly in } H^1(G),\\
	u_n &\to \hat u && \text{in } L^r_{\mathrm{loc}}(G),\ r>2,\\
	u_n(x) &\to \hat u(x) && \text{for a.e.\ } x\in G,
\end{align*}
so that $\hat u\ge0$. Recalling $\tau_n\to1^-$ and $\lambda_n\to\hat\lambda$, we may pass to the limit in the weak formulation of \eqref{b:31} and deduce that $\hat u$ satisfies
\begin{equation}\label{b:54}
	\begin{cases}
		-\hat u'' + W(x)\hat u + \hat\lambda\,\hat u
		= \chi(x)\,\hat u^{\,p-1} & \text{on } G,\\[2mm]
		\displaystyle\sum_{e\succ v}\hat u_e'(v)=0
		& \text{for every vertex } v\in V .
	\end{cases}
\end{equation}
Arguing exactly as in the proof of Proposition \ref{bp:1}, one obtains
\begin{equation}\label{b:55}
	\int_G \bigl|(u_n-\hat u)'\bigr|^2\,dx
	+ \int_G W(x)\,|u_n-\hat u|^2\,dx
	+ \hat\lambda\int_G |u_n-\hat u|^2\,dx \;\longrightarrow\; 0 .
\end{equation}

\medskip
\noindent\textbf{Step 3: $\hat u\not\equiv0$ and $\hat u>0$.}
Suppose $\hat u\equiv0$. Then \eqref{b:55} gives
\begin{equation}\label{b:56}
	\int_G |u_n'|^2\,dx + \int_G W(x)|u_n|^2\,dx
	+ \hat\lambda\int_G |u_n|^2\,dx \;\longrightarrow\; 0 .
\end{equation}
If $\hat\lambda>0$, then by $(W2)$
\begin{align*}
\int_G |u_n'|^2\,dx + \int_G W(x)|u_n|^2\,dx + \hat\lambda\int_G |u_n|^2\,dx \;\ge\; \hat\lambda\int_G |u_n|^2\,dx = \hat\lambda\,\mu >0,
\end{align*}
contradicting \eqref{b:56}. Hence necessarily $\hat\lambda=0$, and \eqref{b:56} yields $\|u_n'\|_{L^2(G)}\to0$. By the Gagliardo--Nirenberg inequality,
\begin{align*}
\|u_n\|_{L^p(G)}^p \le C\,\|u_n\|_{L^2(G)}^{\frac{p+2}{2}}
\|u_n'\|_{L^2(G)}^{\frac{p-2}{2}} \longrightarrow 0 ,
\end{align*}
so that $E_{W,\tau_n}(u_n,G)\to0$. But $E_{W,\tau_n}(u_n,G)=c_{W,\tau_n}\ge\alpha>0$ by Lemma \ref{bl:2}, a contradiction. Therefore $\hat u\not\equiv0$. Since $\hat u\ge0$ solves \eqref{b:54}, the same argument as in Proposition \ref{bp:1}, based on the Kirchhoff conditions and the uniqueness theorem for ordinary differential equations, shows that $\hat u>0$ on $G$.

\medskip
\noindent\textbf{Step 4: $\hat\lambda>0$ and strong convergence.}     Let $\ell$ be an unbounded edge of $G$ satisfying $(W4)$, identified with $[0,+\infty)$.
	Since $\chi\equiv0$ on $\ell$, the restriction of $\hat u$ to $\ell$ satisfies
	\begin{align*}
		-\hat u'' + W(x)\,\hat u + \hat\lambda\,\hat u = 0 \qquad\text{on }(0,+\infty),
	\end{align*}
	with $\hat u>0$ and $\hat u\in H^1(\ell)$.
	Since $\hat\lambda\ge0$ by Step 2, it suffices to exclude $\hat\lambda=0$; in that case $\hat u|_{{\ell}}$ would be a nontrivial solution of \eqref{b:57} in $L^2({\ell})$, which Lemma \ref{bl:7}$(ii)$ forbids. Hence $\hat\lambda>0$.

 All three terms in \eqref{b:55} being nonnegative by $(W2)$, we conclude
\begin{align*}
\|(u_n-\hat u)'\|_{L^2(G)}\to0
\qquad\text{and}\qquad
\|u_n-\hat u\|_{L^2(G)}\to0 ,
\end{align*}
that is, $u_n\to\hat u$ strongly in $H^1(G)$. In particular
$\|\hat u\|_{L^2(G)}^2=\mu$, so $\hat u\in H^1_\mu(G)$, and
$E_{W,1} (\hat u, G)=\lim_{n\to\infty}c_{W,\tau_n}\ge\alpha>0$.
Therefore $(\hat u,\hat\lambda)\in H^1_\mu(G)\times\mathbb{R}^+$ is a positive solution of \eqref{b1:1}--\eqref{b1:3} at a strictly positive energy level. This completes the proof. \qed

\medskip
	      	  \noindent \textbf{Funding} \\
              The second author gratefully acknowledge the financial support provided by the Anusandhan National Research Foundation (ANRF), India, under the project sanction number\\ ANRF/ARGM/2025/002376/MTR.
              
              \medskip
	  \noindent \textbf{Competing Interests} \\
The author declares that there is no conflict of interest. 

	\bibliographystyle{siam}
	\bibliography{ref}

\end{document}